\documentclass[11pt]{amsart}

\usepackage{amsmath,amsthm,amssymb}
\usepackage{booktabs}
\usepackage{array}
\usepackage{multirow}
\usepackage{float}
\usepackage{graphicx}
\usepackage{xcolor}
\usepackage{longtable}
\usepackage[margin=1in]{geometry}
\usepackage[hidelinks]{hyperref}
\theoremstyle{plain}
\newtheorem{thm}{Theorem}[section]
\newtheorem{prop}[thm]{Proposition}
\newtheorem{coro}[thm]{Corollary}
\newtheorem{lem}[thm]{Lemma}

\theoremstyle{remark}
\newtheorem{rem}[thm]{Remark}
\theoremstyle{plain}
\newtheorem*{que}{\textbf{Question}}

\newcommand{\Q}{\mathbb{Q}}
\newcommand{\Z}{\mathbb{Z}}
\newcommand{\R}{\mathbb{R}}

\title[A two-parameter family of rational triangles
       with equal perimeter and area] {Beyond the unique pair:\\
       a two-parameter family of rational triangles\\
       with equal perimeter and area}

 \author[Allen]{Jesse Allen}
\address{D\'epartement de math\'ematiques et de statistique,
Universit\'e de Montr\'eal, Montr\'eal, Qu\'ebec, Canada}
\email{jesse.allen@umontreal.ca}

\author[Beaudry-Ogden]{Olivier Beaudry-Ogden}
\address{D\'epartement de math\'ematiques et de statistique,
Universit\'e de Montr\'eal, Montr\'eal, Qu\'ebec, Canada}
\email{olivier.beaudry-ogden@umontreal.ca}

\author[Bouchard]{Th\'eophane Bouchard}
\address{École Normale Sup\'erieure de Lyon, Lyon, France}
\email{theophane.bouchard@ens-lyon.fr}

\author[Fr\'ely]{Mat\'eo Fr\'ely}
\address{École Normale Sup\'erieure de Lyon, Lyon, France}
\email{mateo.frely@ens-lyon.fr}

\author[Lal\'in]{Matilde Lal\'in}
\address{D\'epartement de math\'ematiques et de statistique,
Universit\'e de Montr\'eal, Montr\'eal, Qu\'ebec, Canada}
\email{matilde.lalin@umontreal.ca}

\author[Oyono-Montoki]{Abel-Jimmy Oyono-Montoki}
\address{D\'epartement de math\'ematiques et de statistique,
Universit\'e de Montr\'eal, Montr\'eal, Qu\'ebec, Canada}
\email{abel-jimmy.oyono-montoki@umontreal.ca}

\author[Ringeling]{Berend Ringeling}
\address{D\'epartement de math\'ematiques et de statistique,
Universit\'e de Montr\'eal, Montr\'eal, Qu\'ebec, Canada}
\email{bjringeling@gmail.com}

\date{\today}

\begin{document}



\maketitle
\begin{abstract}
We study pairs consisting of a rational isosceles triangle and a
rational triangle having a prescribed angle $\theta$, with equal
perimeter and equal area. Writing $\rho=\cos(\theta)\in\Q$, we obtain a
two-parameter description of the solutions and, for fixed $\rho$, a
family of curves $C_\rho$ which are generically of genus $2$. We
determine the two singular specializations: there is no genuine pair
for $\rho=\frac12$, while $\rho=\frac{47}{49}$ admits infinitely many
genuine pairs. We completely parametrize two natural one-parameter
families in which the $\theta$-triangle is isosceles, and prove several
density results for values of $\rho$ admitting one or more genuine
pairs. For thirty-five explicit values of $\rho$, Chabauty--Coleman
computations determine $C_\rho(\Q)$ completely and show that the
corresponding pair is unique up to homothety. We also study multiplicity when one
triangle is fixed; the condition for a scalene $\theta$-triangle to
have two isosceles partners leads to an elliptic $K3$ surface and
yields a dense set of such values of $\rho$.
Finally, we construct infinite families with at least four genuine
pairs and exhibit examples carrying at least five, and in two cases
at least six, genuine pairs.
\end{abstract}

\section{Introduction}
\label{sec:intro}

A classical question in the arithmetic of triangles asks for two non-congruent
triangles, with rational side lengths, that share both their perimeter and their
area. The special case in which one triangle is right and the other is isosceles was posed by Denis Boris as an IBM Research \emph{Ponder This Challenge}  ~\cite{IBMPonderThis2004}. The first systematic answer in a modern form is
due to Hirakawa and Matsumura~\cite{HirakawaMatsumura, HirakawaMatsumura-corr}, who proved that, up to homothety, there is a
\emph{unique} pair consisting of a rational right triangle and a rational
isosceles triangle with equal perimeter and equal area, namely the right triangle
with sides $(135,352,377)$ and the isosceles triangle with sides
$(132,366,366)$. To prove this, they reduce the problem to the determination of the rational points on a genus 2 curve, which is then made precise by
 Chabauty's method. Related work of
Andrica and \c{T}urca\c{s}~\cite{AndricaTurcas} treats pairs sharing other pairs
of invariants (such as the inradius and circumradius, or the semiperimeter and
circumradius), reducing in each case to rational points on a curve of genus $3$
that are found by an elliptic-curve Chabauty computation.

The broader problem of finding non-congruent rational triangles with equal
perimeter and area has also been studied extensively in the setting of Heron
triangles. Kramer and Luca \cite{KramerLuca} and Bremner \cite{Bremner} constructed
parametric families, while van Luijk \cite{vL07} proved, using an elliptic $K3$
surface, that arbitrarily many Heron triangles can have the same perimeter and
area. We emphasize that throughout this paper a rational triangle means a
triangle with rational side lengths; its area need not be rational.

In this paper we return to the perimeter and area setting, and consider
the natural generalization in which the right angle is replaced by an
arbitrary angle.
\begin{que}
Fix $\theta$ with $\cos(\theta)=\rho\in\Q$. Do there exist two non-congruent rational triangles, one isosceles and one having an angle $\theta$, with the same perimeter and area?
\end{que}
We call such a pair a \emph{genuine pair}.
Equating the two invariants leads to
a family of curves
\[
C_\rho\colon y^2=F_\rho(u),
\]
which are generically of genus $2$. The discriminant in the $u$-variable
shows that, inside $-1<\rho<1$, there are exactly two singular
specializations, namely $\rho=\frac12$ and $\rho=\frac{47}{49}$, whose
normalizations have genus $1$. The right-angle setting of Hirakawa and
Matsumura is recovered at $\rho=0$. (Remark~\ref{rem:right-angle}
recovers their classical pair explicitly.)

Our results address both the existence and the uniqueness of genuine
pairs. For $\rho\in\Q\cap(-1,1)$, let $N(\rho)$ denote the number of
distinct genuine pairs, up to homothety, having
$\cos(\theta)=\rho$. The two exceptional values discussed above already
display very different behaviour. For $\rho=\frac12$, corresponding to
an angle of $60^\circ$, there is no genuine pair (Theorem~\ref{thm:60}), while there are infinitely many for
$\rho=\frac{47}{49}$ (Theorem~\ref{thm:47-infinite}). In fact, by Corollary \ref{cor:finite-infinite},
\[
N(\rho)=\infty
\qquad\Longleftrightarrow\qquad
\rho=\frac{47}{49}.
\]
We also show that the supplementary angle of $120^\circ$, equivalent to $\rho=-\frac12$, admits no
genuine pair (Remark~\ref{rem:120}).

We also obtain a two-parameter description of the solutions (Theorem~\ref{thm:param}). For a
fixed parameter $q$, the equations expressing equality of perimeter
and area lead to a singular cubic curve, and intersecting it with a suitable
line gives the parametrization
\[
\rho=\rho(t,q)
=
\frac{q(q+1)+(1-q)(4t^3-t)}
     {(q+1)+(1-q)(4t^3+t)},
\qquad
u=2t(1-\rho).
\]
This parametrization is the starting point for both the theoretical
constructions and the computational searches carried out in the paper.

In contrast to the infinite family at
$\rho=\frac{47}{49}$, we obtain a collection of explicit uniqueness
results. For thirty-five values of $\rho$ arising from a rational
one-parameter family, we determine all the rational points of
$C_\rho$ by Chabauty--Coleman. In each of these cases
$\#C_\rho(\Q)=10$, but exactly one of these rational points gives a
genuine pair. This implies that
\[
N(\rho)=1
\]
for each of the thirty-five values listed in
Table~\ref{tab:rank-one-apex} (Theorem~\ref{thm:rank-one-apex}).

The existence of genuine pairs is a frequent phenomenon. When the
$\theta$-triangle is itself isosceles, two natural families arise,
according to whether $\theta$ is the apex angle or one of the base
angles in the $\theta$-triangle. We call these the apex and base families, respectively. The
apex family admits a complete rational parametrization (Theorem \ref{thm:family-A}), and the
corresponding rational values of $\rho$ are dense in $(-1,1)$ (Corollary \ref{cor:apex-density}).
The base family is also rationally parametrized (Theorem \ref{thm:family-B}). Moreover, the set of
values $\rho\in(0,1)$ belonging simultaneously to both families is
dense in $(0,1)$. In particular, the set of rational values
$\rho\in(0,1)$ for which
\[
N(\rho)\geq2
\]
is dense in $(0,1)$ (Theorem \ref{thm:two-pairs-dense}).

We also study a different form of multiplicity, in which one of the
two triangles is held fixed. A fixed isosceles triangle has at most one
genuine $\theta$-triangle partner, up to congruence (Proposition~\ref{prop:fixed-isosceles}). In the opposite
direction, a fixed scalene $\theta$-triangle can have at most two
genuine isosceles partners, and we obtain an explicit square condition
characterizing when the second partner exists (Proposition~\ref{prop:several-isosceles-partners}). The surface defined by
this square condition is birational to an elliptic $K3$ surface (Proposition~\ref{prop:square-surface-k3}). Using
a section of infinite order on this surface, we prove that the set of
rational $\rho\in(-1,1)$ for which a scalene $\theta$-triangle has two
distinct genuine isosceles partners is dense in $(-1,1)$ (Theorem~\ref{thm:infinitely-many-two-partners}).

There are families with even
 higher multiplicities. We construct a
rational one-parameter family for which every admissible parameter
gives at least four distinct genuine pairs (Theorem~\ref{thm:four-pair-rational}). The resulting values of
$\rho$ are dense in the interval
\[
\left[
\frac{3239209}{6414409},
\frac{85249}{150049}
\right]
\]
(Corollary~\ref{cor:four-pair-density}).
We also give a second construction, governed by an elliptic curve of
rank $1$, which produces infinitely many further rational values of
$\rho$ carrying at least four distinct genuine pairs (Theorem~\ref{thm:four-pair-elliptic}).

Our computational search also reveals examples of still higher
multiplicity (Proposition~\ref{prop:high-multiplicity}). In particular,
\[
N\left(\frac{1952}{3977}\right)\geq6,
\qquad
N\left(\frac{2619}{5069}\right)\geq6,
\]
while
\[
N(\rho)\geq5
\]
for
\[
\rho\in
\left\{
-\frac{327}{1241},
\frac{613}{1495},
\frac{817}{1969},
\frac{7671}{16121},
\frac{302}{527},
\frac{2341}{4141},
\frac{171}{221}
\right\}.
\]
The corresponding pairs are displayed in
Tables~\ref{tab:high-multiplicity-disjoint} and
\ref{tab:high-multiplicity-shared}. These are lower bounds: we do not
claim that the displayed lists exhaust all genuine pairs for these
values of $\rho$.

The arithmetic techniques behind these results vary considerably with
$\rho$. For fixed $\rho$, the problem is generically reduced to the
determination of rational points on a genus-$2$ curve. In the
rank-$1$ cases we use Chabauty--Coleman to obtain complete
determinations of the rational points, but for many specializations in
our computational range a $2$-descent does not prove rank $1$, so
classical Chabauty cannot be applied directly. This difference already
appears in the two isosceles families: the Jacobian of the generic
apex-family curve has rank $1$ over $\Q(t)$ (Proposition~\ref{prop:generic-rank}), whereas the Jacobian of the
generic base-family curve has rank at least $2$ over $\Q(m)$ (Proposition~\ref{prop:base-generic-rank}). Other
parts of the paper use elliptic curves and density of their rational
points, as well as the elliptic $K3$ surface arising from the
second-partner condition. Throughout, one must also separate genuine
solutions from degenerate and coincident configurations.

This paper is organized as follows. In Section~\ref{sec:general-case},
we introduce the general configuration and parametrize the two types of
triangles. Section~\ref{sec:two-geometries} develops two complementary
parametrizations of the equal-perimeter, equal-area condition, leading
both to the genus-$2$ curves $C_\rho$ and to the two-parameter
description of the solutions. In Section~\ref{sec:singular} we study
the two singular specializations $\rho=\frac12$ and
$\rho=\frac{47}{49}$. Section~\ref{sec:isosceles-theta} is devoted to
the apex and base families arising when the $\theta$-triangle is itself
isosceles, together with their parametrizations and density properties.
In Sections~\ref{sec:fixed-triangle} and
\ref{sec:square-condition-surface} we study multiplicity when one
triangle is fixed and the elliptic $K3$ surface arising from the
condition for a second isosceles partner. Section~\ref{sec:jacobian-ranks}
concerns the Mordell--Weil ranks of the Jacobians $J_\rho$, including
the thirty-five rank-$1$ specializations for which we prove uniqueness.
In Section~\ref{sec:four-pairs} we construct families with four genuine
pairs and present examples of higher multiplicity. Finally,
Appendix~\ref{app:code} contains the \textsc{Magma} code used for the
computational results.

\section*{Acknowledgments}
The authors are grateful to Jennifer Balakrishnan for helpful discussions.

\section*{Funding Sources}

BR and ML are grateful to Anthropic for providing access and guidance to \textsc{Claude Max} during the workshop ``AI and number theory'' that took place in May 2026 at the American Institute of Mathematics.

This work was partially supported by the Natural Sciences and Engineering Research Council of Canada (RGPIN-2022-03651 to ML, Undergraduate Student Research Awards to JA and to AJOM), the Fonds de recherche du Qu\'ebec - Nature et technologies (Projet de recherche en \'equipe 345672 to ML,
Suppl\'ements aux bourses de 1er cycle en milieu acad\'emique du CRSNG to AJOM), the  Institut des sciences math\'ematiques (Postdoctoral Fellowship to BR, Bourse de recherche de premier cycle to OBO), the Centre de recherches math\'ematiques (Postdoctoral Fellowship to BR), the \'Ecole Normale Sup\'erieure de Lyon (Stage de recherche to TB and to MF).

\section*{Declaration of Generative AI and AI-Assisted Technologies} During the preparation of this manuscript, the authors used artificial intelligence (AI) tools to assist with mathematical discussions, code writing, and manuscript editing.

Anthropic's \textsc{Claude} contributed to the proof of Proposition~\ref{prop:square-surface-k3} and assisted in the search of examples, particularly those listed in Theorem~\ref{thm:rank-one-apex}.

OpenAI's \textsc{ChatGPT} contributed to the proofs for Propositions~\ref{prop:generic-rank} and \ref{prop:base-generic-rank}, found several of the examples in Theorem~\ref{thm:rank-one-apex} and  Proposition~\ref{prop:high-multiplicity} and was essential to the discovery of Theorems~\ref{thm:four-pair-rational} and \ref{thm:four-pair-elliptic}. It also provided the final versions of the Magma codes and was used  to improve the readability, grammar, and style of the author-drafted text.

The original research questions were entirely directed by the human authors. After using these tools, the authors reviewed, verified, and rewrote all outputs to ensure accuracy and eliminate potential hallucinations. The authors accept full accountability for the content and integrity of the final work.

\section*{Data Availability Statement}

All code needed to verify the computational results in this article
is included in the Appendix.

\section{The general configuration}
\label{sec:general-case}

Let us fix $\theta$ with $\cos(\theta)=\rho\in\Q$ and $-1<\rho<1$. We start by parametrizing the two
families of triangles and continue by imposing the conditions that the perimeters
and areas be equal.

\subsection{A triangle with angle \texorpdfstring{$\theta$}{theta}}
Let a triangle have rational sides $a,b,c$ with the angle $\theta$ between $a$ and
$b$. We will call such a triangle a $\theta$-triangle. We remark that its area is given by $ab\sqrt{1-\rho^2}/2$. By the law of cosines, $c^{2}=a^{2}+b^{2}-2\rho ab$. We write
$x=a/c$, $y=b/c$ and we  intersect the conic $1=x^{2}+y^{2}-2\rho xy$ with the line
$y=qx+1$ through $(0,1)$. After discarding the trivial solution $x=0$, we obtain the parametrization
\[
  x=\frac{2(\rho-q)}{q^{2}-2\rho q+1},\qquad
  y=\frac{1-q^{2}}{q^{2}-2\rho q+1}.
\]
Clearing the common denominator, the triangles of this type are homothetic to those parametrized by
\[
  \tau(q)=\bigl(2(\rho-q),\ 1-q^{2},\ q^{2}-2\rho q+1\bigr),\qquad q\in\Q,
\]
with perimeter $2(1+\rho)(1-q)$ and area $(\rho-q)(1-q^{2})\sqrt{1-\rho^{2}}$.
We remark that this gives a non-degenerate triangle (that is, the three sides have positive length and the triangle inequalities are satisfied) if and only if $|q|<1$ and $\rho>q$.

\subsection{An isosceles triangle}
Now we consider an isosceles triangle with rational sides $2r,s,s$ and such that its area lies in $\sqrt{1-\rho^{2}}\Q$ and therefore has height $h\sqrt{1-\rho^2}$ for some $h\in\Q$ with respect to the side $2r$. By Pythagoras' theorem, $s^2=r^2+(1-\rho^2)h^2$. This leads to an
 analogous construction (intersecting the conic
$1=y^{2}+(1-\rho^{2})x^{2}$ with the line $y=ux+1$), and to the parametrization
\[
  \iota(u)=\bigl(2(1-\rho^{2}-u^{2}),\ u^{2}-\rho^{2}+1,\ u^{2}-\rho^{2}+1\bigr),
  \qquad u\in\Q,
\]
with perimeter $4(1-\rho^{2})$ and area $-2u(1-\rho^{2}-u^{2})\sqrt{1-\rho^{2}}$.
This gives a non-degenerate triangle if and only if $-\sqrt{1-\rho^2}<u<0$.

\section{Two complementary geometries}\label{sec:two-geometries}

\subsection{Equating perimeter and area: the fixed $\rho$ parametrization}
Scaling the $\theta$-triangle by $k\in\Q^{\times}$ and equating the two
perimeters and the two areas gives
\begin{equation}\label{eq:system}
  \begin{cases}
    k(1-q)=2(1-\rho),\\[2pt]
    k^{2}(\rho-q)(1-q^{2})=-2u(1-\rho^{2}-u^{2}).
  \end{cases}
\end{equation}
Eliminating $q$ and excluding the degenerate cases $\rho=1$, $k=0$ leads to the quadratic equation in $k$:
 \[
-2(1-\rho)^2k^2
+\Bigl(
2(1-\rho)^3
+(4-u)(1-\rho)^2
+2u(1-\rho)
-u^3
\Bigr)k
-4(1-\rho)^3
=0.
\]
Since we are looking for solutions with $k\in\Q$, the quadratic equation has a rational solution if and only if its discriminant $F_\rho(u)$ is a rational square.
Thus we reach the curve
\begin{equation}\label{eq:Crho}
  C_{\rho}\colon\quad y^{2}=F_{\rho}(u)=Q_{\rho}(u)\,S_{\rho}(u),
\end{equation}
where
\[
  Q_{\rho}(u)=u^{2}-2(1-\rho)u+\rho^{2}-1
\]
and
\[
  S_{\rho}(u)=u^{4}+2(1-\rho)u^{3}+(1-\rho)(3-5\rho)u^{2}
  -4(1-\rho)^{2}(1+\rho)u-4(1-\rho)^{3}(1+\rho).
\]
A direct discriminant computation gives
\[
 \operatorname{disc}_{u}F_{\rho}
 =-2^{27}(\rho-1)^{21}(\rho+1)^3(2\rho-1)^2(49\rho-47).
\]
Therefore,  for $-1<\rho<1$, the curve $C_{\rho}$ is smooth of genus $2$
except at $
  \rho=\frac12, \rho=\frac{47}{49},$
where its normalizations have genus $1$.  For every smooth specialization we write
$J_{\rho}$ for its Jacobian.
The curve $C_\rho(\Q)$ always has two points at infinity, as well as the
 following trivial points:
\[
  \bigl(0,\ \pm 2(1-\rho)^{2}(1+\rho)\bigr)\quad(\text{degenerate}),
  \qquad
  \bigl(\rho-1,\ \pm 2(\rho-1)^{2}(2\rho-1)\bigr)\quad(\text{coincidence points}),
\]The latter points will be called the \emph{coincidence points}. For
$\rho>0$ they correspond to the configuration
$q=2\rho-1$, $k=1$, where the two triangles coincide. For $\rho\leq0$ this configuration is
non-admissible, but the corresponding rational points of $C_\rho$
are valid and useful.

A solution to the system \eqref{eq:system} will be a  \emph{genuine pair} if $k\,\tau(q)$ and $\iota(u)$ are  non-degenerate
triangles  that
are non-congruent.

\subsection{The complete parametrization with $q$ fixed}
We fix $q \in \Q\setminus \{\pm 1\}$ and consider the system \eqref{eq:system} as we eliminate $k$. This leads to the following cubic equation on $u$:
\begin{equation}\label{eq:Phiq}
  \Phi_{q}\colon\quad
  u^{3}-(1-\rho^{2})\,u-\frac{2(1-\rho)^{2}(\rho-q)(1+q)}{1-q}=0.
\end{equation}
\begin{thm}\label{thm:param}
For every $q\in\Q\setminus\{\pm1\}$, the point $(\rho,u)=(1,0)$ is a
node of $\Phi_q$. The rational points of $\Phi_q$ other than the node
are parametrized by
\begin{equation}\label{eq:param-tq}
  \rho=\rho(t,q)=\frac{q(q+1)+(1-q)(4t^{3}-t)}
                     {(q+1)+(1-q)(4t^{3}+t)},
\end{equation}
for $t\in \Q$ and $(q+1)+(1-q)(4t^3+t)\ne0$.
\end{thm}

\begin{proof} We can
immediately see that at $(\rho,u)=(1,0)$ we have
$\Phi_{q}(1,0)=0$. We also have $\frac{\partial\Phi_{q}}{\partial \rho}=2\rho u-2\frac{\partial}{\partial \rho}\frac{(1-\rho)^{2}(\rho-q)(1+q)}{1-q}$, which is zero at $(1,0)$ since $\rho=1$ is a double root in the second term. Finally we have $\frac{\partial\Phi_{q}}{\partial u}=3u^2-(1-\rho^2)$, which again is zero at $(1,0)$.
The Taylor series around $(\rho,u)=(1,0)$ is given by
\[\Phi_q(\rho,u)=2(\rho-1)(u-(q+1)(\rho-1))+\frac{2(q+1)}{q-1}(\rho-1)^3+(\rho-1)^2u+u^3.\]
Since the degree-two part is a product of two different linear forms,  we have two different tangents and a node. Finally, we intersect the equation $\Phi_q(\rho,u)=0$ with the line $u=2t(1-\rho)$ through $(1,0)$ to get parametrization \eqref{eq:param-tq}. Moreover, $(1,0)$ is the only point of $\Phi_q$ with $\rho=1$.
Hence every rational point $(\rho,u)\ne(1,0)$ lies on a unique line
of the form $u=2t(1-\rho)$, with
$t=\frac{u}{2(1-\rho)}\in\Q.$
Thus \eqref{eq:param-tq} parametrizes all rational points of
$\Phi_q$ other than the node.
\end{proof}
\begin{rem}[The right-angle case]\label{rem:right-angle}
The classical right-angle case is recovered directly from the complete
parametrization. Taking
\[
t=-\frac5{12},
\qquad
q=-\frac{11}{16},
\qquad
\mbox{in \eqref{eq:param-tq} gives}
\qquad
\rho=0,
\qquad
u=-\frac56.
\]
The perimeter equation in \eqref{eq:system} then gives
$k=\frac{32}{27}$, and
\[
k\tau\left(-\frac{11}{16}\right)
=
\left(
\frac{44}{27},\frac58,\frac{377}{216}
\right),
\qquad
\iota\left(-\frac56\right)
=
\left(
\frac{11}{18},\frac{61}{36},\frac{61}{36}
\right).
\]
Multiplying both triangles by $216$ gives
\[
(352,135,377)
\qquad\text{and}\qquad
(132,366,366),
\]
which, up to reordering of the sides, is precisely the pair of
Hirakawa and Matsumura.
\end{rem}

\section{The singular specializations of $C_\rho$} \label{sec:singular}

In this section we treat the two singular specializations of the family
$C_\rho$, namely $\rho=\frac12$ and $\rho=\frac{47}{49}$. In both cases,
the normalization of $C_\rho$ has genus $1$ and is birational over $\Q$
to an elliptic curve. However, the two cases lead to very different
pictures in terms of the number of genuine pairs.

\subsection{The case $\rho=1/2$}

The value $\rho=\frac12$ corresponds to $\theta=\frac{\pi}{3}$. We will show that this case admits no genuine
pair, by reducing it to a single elliptic curve of rank zero.

Recall from Section~\ref{sec:general-case} that we wrote $
C_{\rho}\colon y^{2}=F_{\rho}(u).$
In this case,
\begin{equation}\label{eq:F-half}
  F_{1/2}(u)=\frac{1}{16}(2u+1)^{2}g_{1/2}(u),
  \qquad
  g_{1/2}(u)=(2u-3)(u-1)(2u^{2}+3u+3).
\end{equation}
Setting $ w=\frac{4y}{2u+1}$
gives the birational model
\begin{equation}\label{eq:D}
D_{1/2}\colon\quad w^2=g_{1/2}(u).
\end{equation}
Since $g_{1/2}(u)$ is a squarefree quartic, $D_{1/2}$ is a smooth curve
of genus $1$ and hence gives the normalization of $C_{1/2}$.
 It has, for instance, the rational point $(u,w)=(0,3)$.

\begin{prop}\label{prop:E}
The genus-one curve $D_{1/2}$ is birational over $\Q$ to the elliptic curve
\[
E_{1/2}\colon\quad
Y^2+XY+Y=X^3-X^2-8X+11.
\]
This is the global minimal model of the elliptic curve with LMFDB label
\texttt{90.b3} \cite{lmfdb}. It has conductor $90$, Mordell--Weil rank $0$, and
$E_{1/2}(\Q)_{\mathrm{tors}}\cong \Z/6\Z$,
generated, for example, by the point $(3,1)$.
\end{prop}

\begin{proof}
One can use the point $(0,3)$ to convert the quartic model to a
Weierstrass equation and then further transform it to the global minimal
model.

An explicit birational map $D_{1/2}\dashrightarrow E_{1/2}$ is given by
\[
X=\frac{3(w+3-u)}{2u^2},
\qquad
Y=
\frac{-4u^3-3u^2-3uw-18u+9w+27}{2u^3}.
\]
The inverse map is
\[
u=\frac{3(X-1)}{Y+X+2},
\qquad
w=-3+u+\frac{2Xu^2}{3}.
\]
A direct substitution verifies that these maps are inverse to one another
on their respective domains of definition and carry the equation of $D_{1/2}$
to that of $E_{1/2}$.
\end{proof}

\begin{thm}\label{thm:60}
For $\theta=\frac{\pi}{3}$, that is, for $\rho=\frac12$, there is no
genuine pair.
\end{thm}

\begin{proof}
Let $\overline{D_{1/2}}$ denote the smooth projective model of $D_{1/2}$. By
Proposition~\ref{prop:E},
$E_{1/2}(\Q)\cong \Z/6\Z,$
generated by $P=(3,1)$. Its rational points are
\[
\mathcal O,\qquad
P=(3,1),\qquad
2P=(1,1),\qquad
3P=(-3,1),\qquad
4P=(1,-3),\qquad
5P=(3,-5).
\]
Since the birational map of Proposition~\ref{prop:E} extends to an
isomorphism of the smooth projective models, $\overline{D_{1/2}}(\Q)$
also has six points:
\[
\overline{D_{1/2}}(\Q)
=
\left\{
(0,3),(0,-3),(1,0),
\left(\frac32,0\right),
\infty_+,\infty_-
\right\}.
\]

Let $\overline{C_{1/2}}$ denote the projective closure of $C_{1/2}$.
The six rational points of $\overline{D_{1/2}}$ map to rational
points of $\overline{C_{1/2}}$. Together with the rational singular
point $\left(-\frac12,0\right)$, they give all the rational points
\[
\overline{C_{1/2}}(\Q)
=
\left\{
\infty_+,\infty_-,
\left(0,\frac34\right),
\left(0,-\frac34\right),
(1,0),
\left(\frac32,0\right),
\left(-\frac12,0\right)
\right\}.
\]
To obtain a non-degenerate isosceles triangle, we need
$-\frac{\sqrt{3}}{2}
=
-\sqrt{1-\rho^2}
<u<0.$
The only possible value among those above is $u=-\frac12$. However,
$
u=-\frac12=\rho-1,
$
which is the coincidence point. Indeed, substituting
$\rho=\frac12$ and $u=-\frac12$ into \eqref{eq:system} gives
$q=0$ and $k=1$, and hence
$\iota\left(-\frac12\right) = \tau(0) = (1,1,1)$.
Thus this corresponds to the coincidence of the two equilateral
triangles.

Thus none of the  rational points of $C_{1/2}$ leads to a genuine pair.
\end{proof}

\begin{rem}[The supplementary angle]\label{rem:120}
It is natural to ask about the supplementary angle
$\theta=\frac{2\pi}{3}$, that is, $\rho=-\frac12$. In contrast with
the $60^\circ$ case, this is not a singular specialization:
$C_{-1/2}$ is a smooth curve of genus $2$. In fact,
\[
F_{-1/2}(u)
=
\frac1{16}
(4u^2-12u-3)
(4u^4+12u^3+33u^2-18u-27),
\]
and both factors are irreducible over $\Q$, so $C_{-1/2}$ has no
rational Weierstrass point.

A $2$-descent gives
$\operatorname{rank}J_{-1/2}(\Q)\leq1$, while the coincidence class
has positive canonical height. Hence
$\operatorname{rank}J_{-1/2}(\Q)=1$, and a Chabauty--Coleman
computation gives
\[
C_{-1/2}(\Q)
=
\left\{
\infty_+,\infty_-,
\left(0,\pm\frac94\right),
\left(-\frac32,\pm9\right),
\left(-\frac74,\pm\frac{781}{64}\right)
\right\};
\]
see Appendix~\ref{app:rho-minus-half}. For $\rho=-\frac12$, a
non-degenerate isosceles triangle requires
\[
-\frac{\sqrt3}{2}<u<0.
\]
None of the three affine $u$-coordinates above lies in this interval.
Hence the $120^\circ$ case admits no genuine pair.
\end{rem}

\subsection{The case $\rho=47/49$}

We now consider the second singular specialization
$\rho=\frac{47}{49}.$ In contrast with the case $\rho=\frac12$, we will show that this case
admits infinitely many genuine pairs, by reducing it to a single elliptic
curve of rank one.

At $\rho=\frac{47}{49}$, we have
\begin{equation}\label{eq:F4749}
F_{47/49}(u)
=
\frac{1}{49^6}(49u+8)^2g_{47/49}(u),
\end{equation}
where
\[
g_{47/49}(u)
=
(49u-16)(49u+12)(2401u^2-588u-48).
\]
Setting $
w=\frac{49^3y}{49u+8}$
gives the birational model
\begin{equation}\label{eq:D-47}
D_{47/49}\colon\quad
w^2=g_{47/49}(u).
\end{equation}
The quartic $g_{47/49}(u)$ is squarefree, so $D_{47/49}$ is a
smooth curve of genus $1$  and hence gives the normalization of $C_{47/49}$.
It has, for instance, the rational point
$\left(u,w\right)=\left(\frac{16}{49},0\right)$.

\begin{prop}\label{prop:E47}
The genus-one curve $D_{47/49}$ is birational over $\Q$ to the elliptic
curve
\[
E_{47/49}\colon\quad
Y^2=X^3-348X+2497.
\]
This is the global minimal model of the elliptic curve with LMFDB label
\texttt{1260.d4} \cite{lmfdb}. It has conductor $1260$, Mordell--Weil
rank $1$, and
$E_{47/49}(\Q)_{\mathrm{tors}}\cong\Z/6\Z$.
A generator of the free part is $
G=(6,25)$,
and a generator of the torsion subgroup is
$T=(26,105)$.
\end{prop}

\begin{proof}
One can use the point $\left(\frac{16}{49},0\right)$ to convert the
quartic model to a Weierstrass equation and then further transform it
to the global minimal model.

An explicit birational map $D_{47/49}\dashrightarrow E_{47/49}$ is given by
\[
X=\frac{4(147u-41)}{49u-16},
\qquad
Y=\frac{7w}{(49u-16)^2}.
\]
The inverse map is
\[
u=\frac{4(4X-41)}{49(X-12)},
\qquad
w=\frac{112Y}{(X-12)^2}.
\]
A direct substitution verifies that these maps are inverse to one another
on their respective domains of definition and carry the equation of
$D_{47/49}$ to that of $E_{47/49}$. The arithmetic invariants of $E_{47/49}$
are recorded in the LMFDB \cite{lmfdb}.
\end{proof}

\begin{thm}\label{thm:47-infinite}
For $\rho=\frac{47}{49}$, there are infinitely many distinct genuine
pairs.
\end{thm}

\begin{proof}
Consider the point
$Q=-2G= \left(\frac{276}{25},\frac{101}{125}\right)\in E_{47/49}(\Q)$.
Since $G$ has infinite order, so does $Q$. We will show that $Q$
corresponds to a genuine pair.

Using the inverse birational map in Proposition~\ref{prop:E47}, this
corresponds to
$\left(u,w\right) = \left(-\frac{79}{294},\frac{3535}{36}\right) \in D_{47/49}(\Q)$.
One of the two corresponding $q$-parameters is
$q=\frac{421}{448},$
and \eqref{eq:system} gives $
k=\frac{256}{189}.$
The associated pair is
\[
\iota\left(-\frac{79}{294}\right)
=
\left(
\frac{671}{43218},
\frac{1879}{12348},
\frac{1879}{12348}
\right) \qquad
\mbox{ and }
\qquad
k\tau\left(\frac{421}{448}\right)
=
\left(
\frac{488}{9261},
\frac{869}{5488},
\frac{112879}{1037232}
\right).
\]
Moreover,
$-1<\frac{421}{448}<\frac{47}{49}<1$
and
$
-\frac{8\sqrt3}{49}
=
-\sqrt{1-\left(\frac{47}{49}\right)^2}
<
-\frac{79}{294}
<0.$
Thus both triangles are non-degenerate. The first is isosceles and the
second is scalene, so they are non-congruent. Hence this is a genuine
pair.

We now show that there are infinitely many such pairs. Since
$E_{47/49}(\Q)$ has rank $1$ by Proposition~\ref{prop:E47}, it has
infinitely many rational points. By the Poincar\'e--Hurwitz theorem
\cite[Satz~11, p.~78]{Skolem}, every real neighbourhood of a rational
point of $E_{47/49}$ contains infinitely many rational points of
$E_{47/49}$.

The birational maps above, together with the correspondence of
Sections~\ref{sec:general-case} and \ref{sec:two-geometries}, are regular in a sufficiently small
neighbourhood of $Q$. Since the inequalities defining a genuine pair
are strict, genuineness is an open condition. Thus there is a real
neighbourhood of $Q$ in which every rational point gives a genuine
pair. By the Poincar\'e--Hurwitz theorem, this neighbourhood contains
infinitely many rational points of $E_{47/49}$.

Finally,
$u=\frac{4(4X-41)}{49(X-12)}$
is a nonconstant rational function on $E_{47/49}$ and therefore has
finite fibres. Hence infinitely many of these rational points give
distinct values of $u$, and therefore distinct isosceles triangles and
distinct genuine pairs.

\end{proof}
\begin{coro}\label{cor:finite-infinite}
Let $N(\rho)$ denote the number of distinct genuine pairs, up to
homothety, for a fixed $\rho\in\Q\cap(-1,1)$. Then
\[N(\rho)=\infty \qquad\Longleftrightarrow\qquad \rho=\frac{47}{49}.\]
\end{coro}

\begin{proof}
If
$\rho\notin\left\{\frac12,\frac{47}{49}\right\},
$
then $C_\rho$ is a smooth curve of genus $2$. By Faltings' theorem \cite{Faltings,Faltings-corr},
$C_\rho(\Q)$ is finite, and hence only finitely many genuine pairs can
occur. For $\rho=\frac12$, Theorem~\ref{thm:60} shows that there are
no genuine pairs, whereas Theorem~\ref{thm:47-infinite} gives
infinitely many for $\rho=\frac{47}{49}$.
\end{proof}

\section{When the $\theta$-triangle is isosceles} \label{sec:isosceles-theta}

In this section we consider two important families arising when the
$\theta$-triangle $\tau(q)$ is itself isosceles, but is not homothetic
to $\iota(u)$. We first determine all values of $q$ for which
$\tau(q)$ is isosceles.

\begin{lem}\label{lem:iso-theta}
Let
$\tau(q) = \bigl(2(\rho-q),\,1-q^{2},\,q^{2}-2\rho q+1\bigr)$
be a non-degenerate rational $\theta$-triangle. Apart from the
coincidence case $q=2\rho-1$, the triangle $\tau(q)$ is isosceles if
and only if
\[
q=0\quad\text{with }\rho>0 ,
\qquad\text{or}\qquad
q=1-\sqrt{2(1-\rho)}
\quad\text{with }\sqrt{2(1-\rho)}\in\Q .
\]
\end{lem}

\begin{proof}
Set $
s=\sqrt{2(1-\rho)}.$
Equating the three possible pairs of sides of $\tau(q)$ gives
$q\in \{1-s,\,1+s,\,0,\,\rho,\,2\rho-1,\,-1\}$.
The values $q=\rho$ and $q=-1$ give degenerate triangles, while
$q=1+s$ does not satisfy the admissibility conditions $
|q|<1$.
The value $q=2\rho-1$ is admissible precisely when $\rho>0$ and gives
the coincidence case.

For $q=0$, the admissibility conditions reduce to $\rho>0$.
Finally, if $q=1-s$, then $-1<\rho<1$ implies $0<s<2$, and hence
$-1<q<1$. Moreover, $
\rho-q=\frac{s(2-s)}2>0.$
Thus $q=1-s$ is admissible. This proves the result.
\end{proof}

The family with
$q=1-s$ and $s=\sqrt{2(1-\rho)}$
occurs whenever
$-1<\rho<1$ and $2(1-\rho)\in\Q^{\times2}$.
It corresponds to the case where $\theta$ is the apex angle of an
isosceles $\theta$-triangle. In this case, the equations
of \eqref{eq:system} become
$k=s$ and  $4u^{3}-s^{2}(4-s^{2})u-s^{4}(2-s)^{2}=0.$
The second equation factors as
\begin{equation}\label{eq:apex}
\bigl(2u+s(2-s)\bigr)\bigl(2u^{2}-s(2-s)u-s^{3}(2-s)\bigr)=0.
\end{equation}
The linear factor gives
a pair of congruent triangles. Thus a non-coincidence rational solution
can occur only when the quadratic factor has rational roots, that is, when
$(2-s)(2+7s)\in\Q^{\times 2}$.
Setting $w=\sqrt{(2-s)(2+7s)}$, the two remaining roots are
\begin{equation}\label{eq:rootapex}
u_{\pm}=\frac{s}{4}\bigl(2-s\pm w\bigr).
\end{equation}
Since $0<s<2$, we have $w>2-s$, and hence $u_{+}>0$ and $u_{-}<0$. Therefore only $u_{-}$ can give a
non-degenerate isosceles triangle.

Thus this family of genuine pairs is
governed by the conditions
$2(1-\rho)\in\Q^{\times 2}$ and $ (2-s)(2+7s)\in\Q^{\times 2}$, where
$-1<\rho<1$ and  $\rho\neq\tfrac12.$
We will call this the \emph{apex family}.

The family with $q=0$ occurs when $0<\rho<1$, equivalently when
$0<\theta<\frac{\pi}{2}$. Notice that in this case,  the equations
of \eqref{eq:system} become
$k=2(1-\rho)$ and $u^3-(1-\rho^2)u-2\rho(1-\rho)^2=0$.
The second equation factors as
\begin{equation}\label{eq:base}
\bigl(u+(1-\rho)\bigr)
\bigl(u^2-(1-\rho)u-2\rho(1-\rho)\bigr)=0.
\end{equation}
The first factor gives the coincidence point $u=\rho-1$. Thus a
non-coincidence rational solution can occur only when the quadratic
factor has rational roots, that is, when
$(1-\rho)(1+7\rho)\in\Q^{\times 2}.$
Setting $w=\sqrt{(1-\rho)(1+7\rho)}$, the two remaining roots are
\begin{equation}\label{eq:rootbase}
u_{\pm}=\frac{(1-\rho)\pm w}{2}.
\end{equation}
Since $w^2-(1-\rho)^2=8\rho(1-\rho)>0,$ we have \(w>1-\rho\), and hence \(u_+>0\) and $u_-<0$. Therefore, only $u_-$ can give a non-degenerate isosceles triangle.

Thus this
family of genuine pairs is governed by the conditions
$(1-\rho)(1+7\rho)\in\Q^{\times 2},$ where $0<\rho<1,$ and $ \rho\ne\frac12$.
In this case, $\theta$ is one of the two angles
adjacent to the base of an isosceles $\theta$-triangle. We will call
this the \emph{base family}.

Apart from these two families and the coincidence case $q=2\rho-1$,
every non-degenerate $\theta$-triangle is scalene.

\subsection{The apex family}
Assume that $2(1-\rho)\in \Q^{\times 2}$ and let $s=\sqrt{2(1-\rho)}$. By Lemma~\ref{lem:iso-theta} and the preceding discussion,
this corresponds to the apex family. We now parametrize the
square condition $(2-s)(2+7s)\in\Q^{\times 2}$.
\begin{thm} \label{thm:family-A}
The apex family is completely parametrized by
\begin{equation}\label{eq:apex-param}
s_A(t)=\frac{2(t+1)}{2t^2+t+1},
\qquad
\rho_A(t)=
\frac{4t^4+4t^3+3t^2-2t-1}
     {(2t^2+t+1)^2},
\end{equation}
where
$t\in(-1,0)\cap\Q,$ and $ t\ne-\frac12$.
For every such $t$, these formulas give a genuine apex pair, and every
genuine apex pair arises uniquely in this way.
\end{thm}

\begin{proof}
The conic $
w^2=(2-s)(2+7s)$
contains the rational point $(0,2)$. Intersecting it with the line $
w=2-(4t+1)s$
gives, after removing the solution $s=0$,
\[
s=\frac{2(t+1)}{2t^2+t+1},
\qquad
w=-\frac{4t(t+2)}{2t^2+t+1}.
\]
Hence
\[
\rho=1-\frac{s^2}{2}
=
\frac{4t^4+4t^3+3t^2-2t-1}
     {(2t^2+t+1)^2}.
\]
Taking the root $u_-$ in \eqref{eq:rootapex} gives
$
u=ts^2.$ Since $2t^2+t+1>0$ for every $t\in\R$, the condition
$0<s<2$ is equivalent to $t>-1$ and $t\ne0$. Since $u<0$ is
necessary for a non-degenerate isosceles triangle, we must also have
$t<0$. Thus
$-1<t<0.$
For such $t$,
$
1-\rho^2-u^2
=
\frac{-16t^3(t+1)^2}{(2t^2+t+1)^4}>0,$
so $\iota(u)$ is non-degenerate. Moreover, $q=1-s$ satisfies
$|q|<1$, and
$
\rho-q=\frac{s(2-s)}2>0,
$
so the $\theta$-triangle is also non-degenerate.

It remains to exclude the congruent case. The root of the linear
factor in \eqref{eq:apex} is $-\frac{s(2-s)}2$, and
$
u+\frac{s(2-s)}2
=
\frac{4t(t+1)(2t+1)}{(2t^2+t+1)^2}.
$
Thus, for $-1<t<0$, the quadratic branch meets the congruent branch
only when $t=-\frac12$. At this value,
$
s=1,
\rho=\frac12,
q=0,$ and $
u=-\frac12, $
and both triangles are equilateral. Hence $t=-\frac12$ must be
excluded, while every other rational $t\in(-1,0)$ gives a genuine
apex pair.

Conversely, every genuine apex pair comes from the quadratic factor in
\eqref{eq:apex}, and hence from a rational point on the conic
$w^2=(2-s)(2+7s)$. The parameter $t$ is uniquely determined by
$t=\frac{u}{s^2}=\frac{u}{2(1-\rho)}.$
Therefore every genuine apex pair occurs uniquely in the stated
parametrization.
\end{proof}

\begin{coro}\label{cor:apex-density}
The set of rational values $\rho\in(-1,1)$ admitting a genuine apex
pair is dense in $(-1,1)$.
\end{coro}

\begin{proof}
The function
\[
\rho_A(t)=
\frac{4t^4+4t^3+3t^2-2t-1}
     {(2t^2+t+1)^2}
\]
satisfies
\[
\rho_A'(t)
=
\frac{8t(t+1)(t+2)}
     {(2t^2+t+1)^3}<0
\qquad (-1<t<0),
\]
and
\[
\lim_{t\to-1^+}\rho_A(t)=1,
\qquad
\lim_{t\to0^-}\rho_A(t)=-1.
\]
Hence $\rho_A$ maps $(-1,0)$ homeomorphically onto $(-1,1)$.
Every nonempty open subinterval therefore has a nonempty open
preimage, which contains a rational
$t\ne-\frac12$. By Theorem~\ref{thm:family-A}, the corresponding
rational value $\rho_A(t)$ admits a genuine apex pair.
\end{proof}

\begin{prop}\label{prop:generic-rank}
Let $J^A$ be the Jacobian of the generic apex-family curve $C_{\rho(t)}\colon y^2=F_{\rho(t)}(u) $
over $\Q(t)$, where $
s(t)=\frac{2(t+1)}{2t^2+t+1}$ and $
\rho(t)=1-\frac12s(t)^2.$
Then
$\operatorname{rank}J^A(\Q(t))=1$.
\end{prop}

\begin{proof}
For computational convenience, we set
$n=-\frac1t$ and $A=n^2-n+1$.
Since $\Q(n)=\Q(t)$, it is enough to determine the rank over $\Q(n)$.
In terms of $n$, we have
$s=\frac{2n(n-1)}{n^2-n+2}$.
After the change of variables
\[
X=\frac{(n^2-n+2)^2}{4n(n-1)}u,
\qquad
Y=\frac{(n^2-n+2)^6}{64n^3(n-1)^3}y,
\]
the generic apex curve becomes
\[
\mathcal C_A\colon\quad
Y^2=G_1(X)G_2(X)G_3(X),
\]
where
$G_1(X)=(X-n)(X+1)$,
$G_2(X)=(X+n-1)(X-A)$,
and
$G_3(X)=X^2+AX+n(n-1)A$.
The identities in this proof are verified by direct computation in
Appendix~\ref{app:generic-rank-apex}.
The factorization $G_1G_2G_3$ defines a Richelot isogeny $\phi\colon J^A\longrightarrow \widehat J^A$
over $\Q(n)$; see, for example, \cite[Section~4]{BruinDoerksen}.
Indeed, the determinant of the coefficient matrix of
$G_1,G_2,G_3$ is
$n^2R$,
where
$R=n^4-6n^3+13n^2-13n+7$,
which is nonzero in $\Q(n)$. In particular, the kernel of $\phi$ is
defined over $\Q(n)$.

We now apply the specialization argument of
\cite[Section~5]{Stoll2019}. In particular,
\cite[Remark~5.4]{Stoll2019} states that the isogeny--dual-isogeny
argument of Remark~5.3 applies to genus-$2$ Jacobians admitting a
Richelot isogeny whose kernel is rational over the function field.

For the present family, the discriminant of $G_1G_2G_3$ factors as
\[
\begin{aligned}
\operatorname{disc}_X(G_1G_2G_3)
={}&-n^{12}(n-1)^{12}(n-2)^2(n+1)^2(2n-1)^2\\
&{}\times(n^2-n+1)^3(n^2-n+2)^6\\
&{}\times(3n^2-3n-1)(3n^2-3n+2)^2.
\end{aligned}
\]
For the dual Richelot construction, put
\[
\widehat G_1=[G_2,G_3],
\qquad
\widehat G_2=[G_3,G_1],
\qquad
\widehat G_3=[G_1,G_2],
\]
where
$[F,G]=F'G-FG'$.
Up to the usual scalar factor given by the Richelot determinant, the
dual curve is defined by
$\widehat G_1\widehat G_2\widehat G_3$.
This scalar factor introduces no additional irreducible divisors beyond
those already appearing below. A direct calculation gives
\[
\begin{aligned}
\operatorname{disc}_X
(\widehat G_1\widehat G_2\widehat G_3)
={}&2^6n^{36}(n-1)^6(n-2)(n+1)^4(2n-1)\\
&{}\times(n^2-n+1)^3(n^2-n+2)^3\\
&{}\times(3n^2-3n-1)^2(3n^2-3n+2)R^{12}.
\end{aligned}
\]
Let $K=\Q(n)$. Since the kernels of $\phi$ and $\widehat\phi$
are $K$-rational of order $4$, after choosing generators their
Kummer descent targets may be identified with subgroups of
$(K^\times/K^{\times2})^2$
for each of $\phi$ and $\widehat\phi$. We denote the corresponding
descent maps by $\delta_\phi$ and $\delta_{\widehat\phi}$.

As in the specialization argument of
\cite[Section~5, Remarks~5.3--5.4]{Stoll2019}, the coordinates of
these descent maps are represented by square classes unramified away
from $2$ and the prime divisors in $\Z[n]$ of the discriminants of
the two Richelot models. Hence the images of both descent maps are
contained in a finite product of copies of the subgroup
$
V\subset K^\times/K^{\times2}$
generated by
\[
\begin{gathered}
-1,\ 2,\ n,\ n-1,\ n+1,\ n-2,\ 2n-1,\\
n^2-n+1,\ n^2-n+2,\ 3n^2-3n-1,\ 3n^2-3n+2,\ R.
\end{gathered}
\]
We now specialize at
$n=12$ or equivalently, $t=-\frac1{12}$.
The preceding generators specialize to
\[
-1,\ 2,\ 12,\ 11,\ 13,\ 10,\ 23,\ 133,\ 134,\ 395,\ 398,\ 12091.
\]
Their square classes are linearly independent in
$\Q^\times/\Q^{\times2}$. Indeed,
\[
\begin{gathered}
12\sim3,\qquad
133=7\cdot19,\qquad
134=2\cdot67,\\
395=5\cdot79,\qquad
398=2\cdot199,\qquad
12091=107\cdot113.
\end{gathered}
\]
Thus, in any relation among these square classes, the primes
$3,\ 11,\ 13,\ 23,\ 19,\ 67,\ 79,\ 199,\ 107$
successively force the exponents of
$12,\ 11,\ 13,\ 23,\ 133,\ 134,\ 395,\ 398,\ 12091$
to vanish. The prime $5$ then forces the exponent of $10$ to vanish,
the prime $2$ forces the exponent of $2$ to vanish, and finally the
sign forces the exponent of $-1$ to vanish. Hence specialization at
$n=12$ is injective on $V$, and therefore coordinatewise on the
square-class bounding groups for the Richelot descent and its dual.

Moreover, the Richelot determinant and both discriminants above are
nonzero at $n=12$. Hence the kernels of the Richelot isogeny and its
dual specialize without degeneration. Since each kernel has four
geometric points and all four are already rational over $\Q(n)$,
specialization is surjective on these kernel groups.

Thus the hypotheses of the isogeny--dual-isogeny specialization
argument of \cite[Corollary~5.2 and Remarks~5.3--5.4]{Stoll2019}
are satisfied: specialization is injective on the square-class
bounding groups for the two descents and surjective on the two kernel
groups. Therefore the specialization homomorphism
$J^A(\Q(n))\longrightarrow J^A_{12}(\Q)$
is injective.

At $n=12$ we have
$t=-\frac1{12}$ and $\rho=-\frac{4223}{4489}$.
Let $J^A_{12}$ denote the Jacobian of this specialization. A
$2$-descent gives
$\operatorname{rank}J^A_{12}(\Q)\leq1$.
On the other hand, the divisor class obtained from the coincidence
point has positive canonical height, and hence infinite order.
Therefore
$\operatorname{rank}J^A_{12}(\Q)=1$.
Since specialization is injective, this gives
$\operatorname{rank}J^A(\Q(n))\leq1$.

It remains to prove the opposite inequality. Let
$P= \left( \rho-1,\, 2(\rho-1)^2(2\rho-1) \right)$
be one of the coincidence points on the generic curve, and choose one
of the rational points at infinity, say $\infty_+$. Then
$[P-\infty_+]\in J^A(\Q(n))$.
Its specialization at $n=12$ is the coincidence class considered
above, which has positive canonical height and therefore infinite
order. Consequently $[P-\infty_+]$ cannot be torsion over $\Q(n)$.
Thus
$\operatorname{rank}J^A(\Q(n))\geq1$.

Since $\Q(n)=\Q(t)$, combining the two inequalities gives
\[
\operatorname{rank}J^A(\Q(t))=1.
\]

\end{proof}

\subsection{The base family}

Assume that $q=0$, $0<\rho<1$ with $\rho\neq \frac{1}{2}$, and $(1-\rho)(1+7\rho)\in\Q^{\times 2}$ and let $
w=\sqrt{(1-\rho)(1+7\rho)}$. By Lemma~\ref{lem:iso-theta}, this corresponds to the base family.

\begin{thm}\label{thm:family-B}
The base family is completely parametrized by
\[
\rho_B(m)=\frac{6-2m}{m^2+7},
\qquad
u_B(m)=\frac{(m+1)(m-3)}{m^2+7},
\]
where
$m\in(-1,3)\cap\Q,$ and $ m\ne1$.
For every such $m$, these formulas give a genuine base pair, and every
genuine base pair arises uniquely in this way.
\end{thm}

\begin{proof}
The conic
$w^2=(1-\rho)(1+7\rho)$ contains the rational point $(0,1)$. Intersecting it with the line $w=1+m\rho$ gives
\[
\rho=\frac{6-2m}{m^2+7},
\qquad
w=\frac{(7-m)(m+1)}{m^2+7}.
\]
Taking the root $u_-$ in  \eqref{eq:rootbase} gives
$u=\frac{(m+1)(m-3)}{m^2+7}$.
The conditions $0<\rho<1$ and $w>0$ are equivalent to
$-1<m<3$.
For such $m$, we have $u<0$ and
$1-\rho^2-u^2 = \frac{4(m+1)^3}{(m^2+7)^2}>0$,
so $\iota(u)$ is non-degenerate. Since $\rho>0$, the triangle
$\tau(0)$ is also non-degenerate.

Finally, $u-(\rho-1) = \frac{2(m-1)(m+1)}{m^2+7}$,
so the pair is coincident exactly when $m=1$, corresponding to
$\rho=\frac12$. Excluding this value gives a genuine pair.

Conversely, from \eqref{eq:base}, every genuine base pair must arise from
$u^2-(1-\rho)u-2\rho(1-\rho)=0$.
Choosing the positive
square root $w$ and parametrizing as before gives $m=(w-1)/\rho$ uniquely, so every genuine
base pair occurs uniquely in the stated parametrization.
\end{proof}

\begin{prop}\label{prop:base-generic-rank}
Let $J^B$ be the Jacobian of the generic base-family curve
$C_{\rho(m)}\colon y^2=F_{\rho(m)}(u)$ over $\Q(m)$, where
$
\rho(m)=\frac{6-2m}{m^2+7}.
$
Then
$\operatorname{rank}J^B(\Q(m))\geq2$.

\end{prop}

\begin{proof}
In addition to the genuine base point of Theorem~\ref{thm:family-B},
\[
P_B(m)=
\left(
\frac{(m+1)(m-3)}{m^2+7},
\frac{2(m-1)(m+1)^4(m+5)}{(m^2+7)^3}
\right),
\]
the curve has the coincidence point
\[
P_{\mathrm c}(m)=\left(
-\frac{(m+1)^2}{m^2+7},
-\frac{2(m-1)(m+1)^4(m+5)}{(m^2+7)^3}
\right).
\]
Choosing one of the rational points at infinity, say $\infty_+$, we obtain the divisor classes in $J^B(\Q(m))$
$D_B(m)=[P_B(m)-\infty_+],$ and $D_{\mathrm c}(m)=[P_{\mathrm c}(m)-\infty_+]$.
We prove that these two classes are independent by specializing at
$m=\frac73$ and $\rho=\frac3{28}$.
At this specialization,
\[
P_B=
\left(
-\frac5{28},
\frac{6875}{5488}
\right),
\qquad
P_{\mathrm c}=
\left(
-\frac{25}{28},
-\frac{6875}{5488}
\right).
\]
The specialized curve is smooth. A canonical-height computation in
\textsc{Magma} (see Appendix~\ref{app:generic-rank-base}) gives
\[
\widehat h(D_{\mathrm c})
=
0.8666169291607799\ldots,
\qquad
\widehat h(D_B)
=
1.9604493795720781\ldots,
\]
and
$\langle D_{\mathrm c},D_B\rangle = 0.5827133855804506\ldots$.
Hence the determinant of their N\'eron--Tate height-pairing matrix is
\[
\widehat h(D_{\mathrm c})\widehat h(D_B)
-
\langle D_{\mathrm c},D_B\rangle^2
=
1.3594037313652797\ldots>0.
\]
Thus the specialized classes are independent modulo torsion.

If $D_{\mathrm c}(m)$ and $D_B(m)$ were dependent modulo torsion in
$J^B(\Q(m))$, then, after multiplying by the order of the torsion
class, we would obtain a nontrivial integral relation between them.
Specializing at $m=7/3$ would give the same relation between the two
specialized classes, contradicting the positive determinant above.
Therefore $D_{\mathrm c}(m)$ and $D_B(m)$ are independent in
$J^B(\Q(m))$ modulo torsion, and we conclude that
$\operatorname{rank}J^B(\Q(m))\geq2$.

\end{proof}

\subsection{The intersection of both families}

\begin{thm}\label{thm:two-pairs-dense}
The set of rational values $\rho\in(0,1)$ having at least two
distinct genuine pairs is dense in $(0,1)$. More precisely, these
values may be chosen so that one genuine pair belongs to the apex
family and the other belongs to the base family.
\end{thm}

\begin{proof}
Recall from Theorem~\ref{thm:family-A} that the apex family is
parametrized by $\rho_A(t)$ satisfying \eqref{eq:apex-param}. Such a value also belongs to the base family precisely
when
$(1-\rho_A(t))(1+7\rho_A(t))$
is a rational square. A direct substitution gives
\[
(1-\rho_A(t))(1+7\rho_A(t))
(2t^2+t+1)^4
=
4(t+1)^2
\bigl(16t^4+16t^3+13t^2-6t-3\bigr).
\]
Thus the simultaneous apex--base condition is governed by the genus-one curve
\[
\mathcal E:\qquad
v^2=16t^4+16t^3+13t^2-6t-3.
\]
This quartic is birational over $\Q$ to the elliptic curve
\[
E:\qquad Y^2+XY+Y=X^3+2X+32.
\]
An explicit birational map $\mathcal E\dashrightarrow E$ is given by
\[
X=8t^2+4t-2v+1,
\qquad
Y=-32t^3-28t^2+8tv-15t+3v+2,
\]
with inverse
\[
t=\frac{7-3X-2Y}{8X+10},
\qquad
v=4t^2+2t+\frac{1-X}{2}.
\]
A direct substitution verifies that these maps carry the equation of
$\mathcal E$ to that of $E$ on their respective domains of definition.

The displayed equation is the global minimal model of the elliptic
curve with LMFDB label \texttt{294.c2}~\cite{lmfdb}. It has
Mordell--Weil rank $1$ and discriminant
$\Delta_E=-444528<0$.
Thus $E(\R)$ is connected, and since $E(\Q)$ has positive rank, its
rational points are dense in $E(\R)$ by the Poincar\'e--Hurwitz
theorem~\cite[Satz~11, p.~78]{Skolem}. Equivalently,
$\mathcal E(\Q)$ is dense in $\mathcal E(\R)$.

We now determine the range of the corresponding values of $\rho$.
Differentiating gives
$\rho_A'(t) = \frac{8t(t+1)(t+2)} {(2t^2+t+1)^3}$.
Thus
$\rho_A'(t)<0 \qquad (-1<t<0)$.
Since
$\rho_A(-1)=1, \qquad \rho_A(0)=-1$,
there is a unique $t_0\in(-1,0)$ for which $\rho_A(t_0)=0$, and
$\rho_A:(-1,t_0)\longrightarrow(0,1)$
is a homeomorphism.

For every $t\in(-1,t_0)$ the identity above gives
$16t^4+16t^3+13t^2-6t-3>0$,
because $0<\rho_A(t)<1$. Hence the real curve $\mathcal E(\R)$ has
points above every $t\in(-1,t_0)$.

Let $U\subset(0,1)$ be a nonempty open interval. Then
$I_U:=\rho_A^{-1}(U)$
is a nonempty open subinterval of $(-1,t_0)$. Removing
$t=-\frac12$ if necessary still leaves a nonempty open set. The
corresponding subset
$\{(t,v)\in\mathcal E(\R):t\in I_U,\ t\ne-\tfrac12\}$
is therefore nonempty and open. It contains infinitely
many rational points, since the rational points are dense.

For any such rational point $(t,v)$, the value
$\rho=\rho_A(t)$ is rational and, by
Theorem~\ref{thm:family-A}, gives a genuine apex pair. Furthermore,
\[
(1-\rho)(1+7\rho)
=
\left(
\frac{2(t+1)v}{(2t^2+t+1)^2}
\right)^2,
\]
so $\rho$ also satisfies the base-family condition. Since
$0<\rho<1$ and $\rho\ne\frac12$, Theorem~\ref{thm:family-B} gives a
genuine base pair.

These two pairs are distinct: in the first the angle $\theta$ is the
apex angle of the isosceles $\theta$-triangle, whereas in the second
it is a base angle. A non-degenerate isosceles triangle can have
$\theta$ both as an apex angle and as a base angle only in the
equilateral case $\rho=\frac12$, which has been excluded.

Finally, each value of $t$ gives at most two points $(t,\pm v)$ on
$\mathcal E$, while $\rho_A$ is strictly monotone on $(-1,t_0)$.
Thus the infinitely many rational points above $I_U$ yield infinitely
many distinct values of $\rho$ in $U$. This proves the result.
\end{proof}

\section{Multiplicity when one triangle is fixed} \label{sec:fixed-triangle}

In this section we study multiplicity from the point of view of one
fixed triangle, rather than only the angle. That is, we ask how many triangles of the other type can have
the same perimeter and area.

\subsection{Fixed $\theta$-triangle}

Fix $\rho$ and $q$ such that $\tau(q)$ is non-degenerate. Recall that the scaling factor is fixed by
$k=\frac{2(1-\rho)}{1-q}$ from the equality of perimeters.

Suppose that $k\tau(q)$ and $\iota(u_0)$ form a genuine pair. By
\eqref{eq:param-tq}, there is a uniquely determined $t\in\Q$ such that $u_0=2t(1-\rho).$

We first treat the case in which the fixed $\theta$-triangle is
scalene.

\begin{prop}\label{prop:several-isosceles-partners}
Suppose that $k\tau(q)$ is scalene and that
$\bigl(k\tau(q),\iota(u_0)\bigr)$ is a genuine pair. Put
\[
\Delta_{t,q}
=
-(q-1)(q+2t+1)
\bigl(
3q^2t^2+q^2-2qt^3+2q+2t^3-3t^2+1
\bigr).
\]
Then the fixed triangle $k\tau(q)$ has a second genuine rational
isosceles partner if and only if
$\Delta_{t,q}\in\Q^{\times2}$.
When this condition holds, the second partner is
$\iota(u_1)$, where
\[
u_1=
\frac{
-u_0-\sqrt{\,4(1-\rho^2)-3u_0^2\,}
}{2}.
\]
In particular, a fixed scalene $\theta$-triangle has at most two
genuine rational isosceles partners.
\end{prop}

\begin{proof}
Since $u_0$ is a root of the cubic $\Phi_q$ in
\eqref{eq:Phiq}, we have
\[
\Phi_q(u)
=
(u-u_0)
\left(
u^2+u_0u+u_0^2-(1-\rho^2)
\right).
\]
Thus the two remaining roots are rational if and only if
$\delta := 4(1-\rho^2)-3u_0^2 \in \Q^{\times 2}$.
Substituting
$\rho=\rho(t,q)$ and $
u_0=2t(1-\rho),
$ gives
\[
\delta
=
\frac{
4\Delta_{t,q}
}{
\bigl(4qt^3+qt-q-4t^3-t-1\bigr)^2
}.
\]
Hence $\delta$ is a nonzero rational square if and only if
$\Delta_{t,q}$ is a nonzero rational square.

Assume this condition holds. The two remaining roots are
$u_{\pm} = \frac{-u_0\pm\sqrt{\delta}}{2}$.
Because the original pair is genuine,
$-\sqrt{1-\rho^2}<u_0<0$.
It follows immediately that $u_+>0$, so $u_+$ does not lead to a non-degenerate isosceles triangle.

For the other root, since
$\delta-u_0^2 = 4(1-\rho^2-u_0^2)>0
$,  we have $\sqrt{\delta}>-u_0$, and hence $u_-<0$. Moreover, one can directly check that
$(-u_0+2\sqrt{1-\rho^2})^2>4(1-\rho^2)-3u_0^2$.
Therefore
$u_- >-\sqrt{1-\rho^2}$ and
$\iota(u_-)$ is non-degenerate.

Finally, $k\tau(q)$ is scalene whereas $\iota(u_-)$ is isosceles, so
the two triangles cannot be congruent. Thus $u_-=u_1$ gives a second
genuine pair. Conversely, any second rational isosceles partner must
come from one of the two remaining roots of $\Phi_q$, so the square
condition is necessary.
\end{proof}
\begin{rem}
The assumption that the fixed $\theta$-triangle is scalene is
essential. If $k\tau(q)$ is isosceles and $u_0$ gives a genuine pair, then the other
admissible negative root of $\Phi_q$ corresponds to the congruent pair.
Thus an
isosceles $\theta$-triangle does not acquire a second genuine partner
from Proposition~\ref{prop:several-isosceles-partners}.
\end{rem}

\subsection{Fixed isosceles triangle}

We now consider the reverse question. Fix $\rho$ and a non-degenerate isosceles triangle $\iota(u)$, and ask how many $\theta$-triangles can have the same perimeter and area. In contrast with the situation of Proposition~\ref{prop:several-isosceles-partners}, there is no multiplicity up to congruence.

\begin{prop}\label{prop:fixed-isosceles}
Fix $-1<\rho<1$ and $ -\sqrt{1-\rho^2}<u<0$. Suppose that $(k\tau(q), \iota(u))$ is a genuine pair.  Then any other solution $k'\tau(q')$ for the same fixed isosceles triangle is congruent to $k\tau(q)$. More precisely, the only possible second set of parameters is
\[
q'=\frac{2\rho-1-q}{1-q},
\qquad
k'=1-q,
\]
and
\[
k'\tau(q')
=
\bigl(
k(1-q^2),\,
2k(\rho-q),\,
k(q^2-2\rho q+1)
\bigr).
\]
Thus $k'\tau(q')$ is obtained from $k\tau(q)$ simply by interchanging the first two sides. In particular, a fixed isosceles triangle has at most one genuine $\theta$-triangle partner up to congruence.
\end{prop}

\begin{proof}
From the perimeter equation in \eqref{eq:system}, we have
$k=\frac{2(1-\rho)}{1-q}.$
Substituting this into the area equation in \eqref{eq:system} gives the following quadratic equation in $q$:
\[
2(1-\rho)^2(\rho-q)(1+q)
+u(1-\rho^2-u^2)(1-q)=0.
\] Hence there are at most two possible values of $q$.

Suppose that $q$ is one root. Using that the equation is satisfied by $q$, the other root can be expressed as
$q'=\frac{2\rho-1-q}{1-q}$.
The corresponding scaling factor is
given by $k'=\frac{2(1-\rho)}{1-q'}=1-q$.
A direct substitution gives
\[
k'\tau(q')
=
\bigl(
k(1-q^2),\,
2k(\rho-q),\,
k(q^2-2\rho q+1)
\bigr),
\]
whereas
\[
k\tau(q)
=
\bigl(
2k(\rho-q),\,
k(1-q^2),\,
k(q^2-2\rho q+1)
\bigr).
\]
Thus the two scaled $\theta$-triangles differ only by interchanging the two sides adjacent to the angle $\theta$, and are therefore congruent.
\end{proof}

\begin{rem}
The transformation
$q\longmapsto\frac{2\rho-1-q}{1-q}$
is an involution. Its admissible fixed point satisfies
$q=1-\sqrt{2(1-\rho)}$,
which is precisely the apex-isosceles case of Lemma~\ref{lem:iso-theta}.
\end{rem}

\section{The square-condition surface}
\label{sec:square-condition-surface}

Proposition~\ref{prop:several-isosceles-partners} shows that the problem of finding a second isosceles partner for a fixed scalene $\theta$-triangle is governed by a square condition. Consider the surface
\[
S:\qquad w^2=\Delta_{t,q}=-(q-1)(q+2t+1)
\bigl(
3q^2t^2+q^2-2qt^3+2q+2t^3-3t^2+1
\bigr).
\]
A rational point $(t,q,w)$ on $S$ gives a parameter pair $(t,q)$ for which all three roots of the cubic $\Phi_q$ are rational. Under the hypotheses of Proposition~\ref{prop:several-isosceles-partners}, a rational point with $w\ne0$ is precisely the condition that the fixed scalene $\theta$-triangle admit a second genuine rational isosceles partner.
In this section we study the surface $S$.
\begin{prop}\label{prop:square-surface-k3}
The surface $S$ is birational over $\Q$ to a relatively minimal elliptic
surface
\[
\pi:E\longrightarrow\mathbb P^1_t
\]
whose generic fibre is given by
\[
y^2=(x-c(t))(x^2-4d(t)),
\]
where $c(t)=2t(t^3-2t^2-3t-2)$ and
$d(t)=t^3(t-4)(t-1)^2(t+1)^2$. The surface $E$ is a $K3$ surface.
\end{prop}

\begin{proof}
Put
\[
\alpha=-8(t+1),\qquad
\beta=2(t+1)(t^3-3t^2-2).
\]
The birational change of variables
\[
x=\beta+\frac{\alpha}{q-1},
\qquad
y=\frac{\alpha w}{(q-1)^2},
\]
with inverse
\[
q=1+\frac{\alpha}{x-\beta},
\qquad
w=\frac{\alpha y}{(x-\beta)^2},
\]
transforms $w^2=\Delta_{t,q}$ into
\[
y^2=(x-c(t))(x^2-4d(t)).
\]

The discriminant of this Weierstrass model is
\[
2^{16}t^7(t-4)(t-1)^2(t+1)^2(2t+1)^4.
\]
A calculation with Tate's algorithm
(see Appendix~\ref{app:square-condition}) shows that the singular fibres
are precisely
\[
I_1^*\ \text{at }t=0,\qquad
I_8\ \text{at }t=\infty,\qquad
I_4\ \text{at }t=-\frac12,\qquad
I_2\ \text{at }t=\pm1,\qquad
I_1\ \text{at }t=4.
\]
Their Euler numbers add up to $7+8+4+2+2+1=24$. Hence
$e(E)=24=12\chi(\mathcal O_E)$, so $\chi(\mathcal O_E)=2$. (See, for
example, \cite[Sections~6.6--6.7]{SchuttShioda}.)

Since the fibration has a section, it has no multiple fibres, and the
canonical bundle formula gives
\[
\omega_E\simeq
\pi^*\bigl(\omega_{\mathbb P^1}\otimes\mathcal L\bigr),
\qquad
\deg\mathcal L=\chi(\mathcal O_E)=2.
\]
Thus $\mathcal L\simeq\mathcal O_{\mathbb P^1}(2)$ and
$\omega_E\simeq\mathcal O_E$. By Serre duality,
$h^2(E,\mathcal O_E)=1$, and hence
$2=\chi(\mathcal O_E)=1-h^1(E,\mathcal O_E)+1$. Therefore
$h^1(E,\mathcal O_E)=0$, and $E$ is a $K3$ surface.
\end{proof}

The elliptic fibration has the following useful rational sections.

\begin{lem}\label{lem:k3-sections}
The points
\[
P=
\left(
2t^2(t+1)(t+5),\,
8t^2(t+1)(2t+1)^2
\right) \qquad \mbox{ and }\qquad
P_4=
\left(
2t^2(t-1)^2,\,
8t^2(t-1)(2t+1)
\right)
\]
define rational sections of $\pi:E\to\mathbb P^1_t$. Moreover,
$2P_4=(c(t),0)$, so $P_4$ has order $4$ on the generic fibre.
\end{lem}

\begin{proof}
One can see by direct computation that $P$ and $P_4$ lie on the generic fibre. The duplication formula gives $2P_4=(c(t),0)$, which is
a nontrivial point of order $2$, and hence $P_4$ has order $4$.
\end{proof}

\begin{lem}\label{lem:k3-positive-rank}
The section $P$ has infinite order on the generic fibre of $E$.
Therefore, for all but finitely many $t_0\in\Q$ for which the fibre
$E_{t_0}$ is smooth, the specialization $P(t_0)$ has infinite order.
In particular, $\operatorname{rank}E_{t_0}(\Q)>0$ for all but finitely
many such $t_0$.
\end{lem}

\begin{proof}
Specializing at $t=2$ gives
\[
E_2:\qquad y^2=(x+32)(x^2+576),
\qquad
P(2)=(168,2400).
\]
The change of variables $x=4X-12$, $y=8Y$ gives the model
$Y^2=X^3-X^2+15X+225$, with LMFDB label
\texttt{480.c4}~\cite{lmfdb}.

The curve has good reduction at $7$ and $13$, and a direct count
(see Appendix~\ref{app:square-condition}) gives
$\#E_2(\mathbb F_7)=12$ and $\#E_2(\mathbb F_{13})=16$. Thus any
rational torsion point has order dividing $4$. On the other hand,
$2P(2)=(32,320)$, which is neither the identity nor a point of order
$2$. Hence $P(2)$ has infinite order, and therefore the section $P$
has infinite order on the generic fibre. The result now follows from
Silverman's specialization theorem~\cite{Silverman1983}.
\end{proof}

\begin{rem}
An elliptic $K3$ surface also appears in van Luijk's study
\cite{vL07} of Heron triangles with common perimeter and area. The two
surfaces are different. Indeed, the reducible fibres of the elliptic
fibration above contribute rank
$5+7+3+1+1=17$
to the trivial lattice, which therefore has rank $19$. Since the
section $P$ of Lemma~\ref{lem:k3-sections} has infinite order by
Lemma~\ref{lem:k3-positive-rank}, the Shioda--Tate formula gives
$\rho(E)\geq20$. Hence $\rho(E)=20$, the maximal possible Picard
number for a complex $K3$ surface. By contrast, van Luijk's $K3$
surface has Picard number $18$.
\end{rem}

\begin{thm}\label{thm:infinitely-many-two-partners}
The set of rational values $\rho\in(-1,1)$ for which a scalene
rational $\theta$-triangle has two distinct genuine rational isosceles
partners with the same perimeter and area is dense in $(-1,1)$.
\end{thm}

\begin{proof}
Let $U\subset(-1,1)$ be a nonempty open interval. For $t>0$, put
\[
\eta(t)=\rho(t,-1)=\frac{4t^2-1}{4t^2+1}.
\]
The map $\eta:(0,\infty)\to(-1,1)$ is a homeomorphism. Hence we may
choose $t\in\Q_{>0}$ such that $\eta(t)\in U$, the fibre $E_t$ is
smooth, $P(t)$ has infinite order, and $t\ne\frac12$. This is possible
by Lemma~\ref{lem:k3-positive-rank}, since only finitely many rational
values of $t$ need to be avoided.

Although $q=-1$ is not admissible for the geometric problem, since
$\tau(-1)$ is degenerate, it is a convenient boundary point of the
square-condition curve. We use it only as a base point and then move
to admissible values $q>-1$ nearby.
At $q=-1$ we have $\Delta_{t,-1}=16t^4$, so
$(t,-1,-4t^2)$ is a rational point of the square-condition surface.
Under the birational map of Proposition~\ref{prop:square-surface-k3},
it corresponds to
\[
B_t=\bigl(4t+c(t),\,8t^2(t+1)\bigr)\in E_t(\Q).
\]
We claim that $B_t\in E_t(\R)^0$, the connected component containing the identity. If $0<t<4$, then $d(t)<0$, apart
from the singular value $t=1$, so $E_t(\R)$ is connected. If $t>4$,
then $d(t)>0$, $c(t)>0$, and
$c(t)^2-4d(t)=16t^2(2t+1)^2>0$. Thus $c(t)$ is the largest real root
of the cubic. Since $x(B_t)-c(t)=4t>0$, the point $B_t$ lies on the
unbounded component, namely $E_t(\R)^0$.

For any real elliptic curve, the quotient
$E_t(\R)/E_t(\R)^0$ has order at most $2$, so the double of any real point
belongs to the identity component. Hence $2P(t)\in E_t(\R)^0$ and has
infinite order. Since $E_t(\R)^0$ is topologically a circle,
$\langle 2P(t)\rangle$ is dense in $E_t(\R)^0$.

We next work near $(q,w)=(-1,-4t^2)$. Since $w\ne0$, the equation
$w^2=\Delta_{t,q}$ defines a real branch with $q$ as a local
coordinate. For $q>-1$ sufficiently close to $-1$, continuity gives
\[
-1<q<\rho(t,q)<1,\qquad
\rho(t,q)\in U,\qquad
\Delta_{t,q}>0.
\]
After shrinking the interval if necessary, the corresponding
$\theta$-triangle is scalene. Indeed, by Lemma~\ref{lem:iso-theta},
the only non-degenerate isosceles cases are
$q=0$, $q=1-\sqrt{2(1-\rho)}$, and $q=2\rho-1$. None can accumulate
at $q=-1$, except possibly the last when $\eta(t)=0$, which would
give $t=\frac12$, and this value was excluded.

The corresponding points with $q>-1$ sufficiently close to $-1$
form a nonempty open arc $\mathcal N\subset E_t(\R)^0$. By the
density of $\langle2P(t)\rangle$, the arc $\mathcal N$ contains
infinitely many rational points. For every such point, the
corresponding $q$ and $w$ are rational and satisfy
$w^2=\Delta_{t,q}$.

Fix one of these rational points and write
$\rho=\rho(t,q)$ and $u_0=2t(1-\rho)>0$. Since $u_0$ is a root of
$\Phi_q$, we have
\[
\Phi_q(u)
=
(u-u_0)
\left(
u^2+u_0u+u_0^2-(1-\rho^2)
\right).
\]
The identity computed in the proof of
Proposition~\ref{prop:several-isosceles-partners} shows that the
discriminant of the quadratic factor is a nonzero rational square
whenever $\Delta_{t,q}$ is. Let $u_1,u_2\in\Q$ be its two distinct
roots.

Since $-1<q<\rho<1$, every root $u$ of $\Phi_q$ satisfies
\[
u\bigl(u^2-(1-\rho^2)\bigr)
=
\frac{2(1-\rho)^2(\rho-q)(1+q)}{1-q}>0.
\]
Applying this first to $u_0>0$ gives $u_0^2>1-\rho^2$. Therefore
\[
u_1+u_2=-u_0<0,
\qquad
u_1u_2=u_0^2-(1-\rho^2)>0,
\]
so $u_1,u_2<0$. Applying the same identity to each $u_i$ then gives
$u_i^2<1-\rho^2$, and hence
\[
-\sqrt{1-\rho^2}<u_i<0
\qquad (i=1,2).
\]
Thus $\iota(u_1)$ and $\iota(u_2)$ are two distinct non-degenerate
rational isosceles triangles. Both have the same perimeter and area
as the same scaled $\theta$-triangle $k\tau(q)$, and since this
$\theta$-triangle is scalene, both pairs are genuine.

Finally, for fixed $t$ the inverse birational map is
$q=1+\alpha/(x-\beta)$, so $q$ has finite fibres on $E_t$. Hence the
infinitely many rational points of $\mathcal N$ give infinitely many
rational values of $q$. For fixed $t$, the equation
$\rho(t,q)=r$ is quadratic in $q$, so each value of $\rho$ arises from
at most two values of $q$. Therefore infinitely many distinct rational
values of $\rho$ occur in $U$. Since $U$ was arbitrary, the desired
set is dense in $(-1,1)$.
\end{proof}

\begin{rem}
Theorem~\ref{thm:infinitely-many-two-partners} also contains explicit
one-parameter subfamilies.
Applying the inverse birational map of Proposition~\ref{prop:square-surface-k3} to the section $P$ gives $q=\rho=\frac{4t^2-1}{4t^2+1},$
which corresponds to a degenerate $\theta$-triangle. Applying the same
map to $2P$ instead gives
\[
q_t=\frac{4t^2-2t-1}{4t^2+1}
\qquad \mbox{
and }\qquad
\rho_t=
\frac{16t^4+16t^3+16t^2-12t-3}
     {(4t^2+1)(4t^2+4t+5)}.
\]
For $t\in\Q$, $t>\frac14$, put
\[
\begin{aligned}
T_t={}&\bigl(
(4t^2+1)^2,\,
t(4t-1)(4t^2+4t+5),\,
(t+1)(20t^2-4t+1)
\bigr),\\
I_{1,t}={}&\bigl(
4t(t+1)(4t-1),\,
16t^4+8t^3+14t^2-2t+1,\,
16t^4+8t^3+14t^2-2t+1
\bigr),\\
I_{2,t}={}&\bigl(
4t(t+1)(4t^2+1),\,
8t^4+8t^3+18t^2-6t+1,\,
8t^4+8t^3+18t^2-6t+1
\bigr).
\end{aligned}
\]
A direct calculation shows that $T_t$ has an angle with cosine
$\rho_t$, and that all three triangles have perimeter
\[
2(16t^4+16t^3+20t^2-4t+1)
\]
and squared area
\[
4t^2(t+1)^2(4t-1)^2(4t^2+1)^2
(16t^4+16t^3+20t^2-4t+1).
\]
For $t>\frac14$ all three triangles are non-degenerate, and
$I_{1,t}$ and $I_{2,t}$ are distinct isosceles triangles. Moreover,
for rational $t>\frac14$, the triangle $T_t$ is scalene unless
$t=\frac12$ or $t=\frac32$. Thus, for
\[
t\in\Q,\qquad t>\frac14,\qquad
t\notin\left\{\frac12,\frac32\right\},
\]
the pairs $(T_t,I_{1,t})$ and $(T_t,I_{2,t})$ are two distinct genuine
pairs with the same value $\rho=\rho_t$.
\end{rem}

\section{Jacobian ranks}
\label{sec:jacobian-ranks}

We now investigate the Mordell--Weil ranks of the Jacobians
$J_\rho$. The apex family is particularly suitable for a systematic
computation, since it is parametrized by a single rational parameter
$t\in(-1,0)$. Proposition~\ref{prop:generic-rank} shows that the
generic Jacobian in this family has rank $1$ over $\Q(t)$, but the rank
may increase after specialization. We computed the
$2$-descent rank bound for every reduced parameter
\begin{equation} \label{eq:198}
-1<t<0,\qquad t\ne-\frac12,\qquad
\operatorname{den}(t)\leq25,
\end{equation}
where $\operatorname{den}(t)$ denotes the denominator of $t$ when written in lowest terms.

Exactly thirty-five of these specializations have rank bound $1$;
they are the values listed in Table~\ref{tab:rank-one-apex} below.
The same search also produces specializations with larger rank bounds,
some of which are investigated in Proposition~\ref{prop:apex-higher-ranks}.
The computation is described in
Appendix~\ref{app:apex-rank-census}.

\subsection{Rank-one specializations and uniqueness}

\begin{thm}\label{thm:rank-one-apex}
For each of the thirty-five values of $\rho$ listed in
Table~\ref{tab:rank-one-apex},
\[
\operatorname{rank}J_\rho(\Q)=1
\qquad\text{and}\qquad
\#C_\rho(\Q)=10.
\]
Moreover, exactly one rational point of $C_\rho$ gives a genuine pair.
Consequently, the apex pair displayed in
Table~\ref{tab:rank-one-apex} is the unique genuine pair for each of
these values of $\rho$.
\end{thm}

\begin{table}[H]
\centering
\setlength{\tabcolsep}{4pt}
\renewcommand{\arraystretch}{1.12}

\begin{tabular}{c|c|c||c|c|c}
\hline
$t$ & $\rho$ & pair
&
$t$ & $\rho$ & pair
\\
\hline

$-\frac{1}{18}$ &
$-\frac{11551}{11858}$ &
$\begin{array}{c}
(77,77,153)\\
(9,149,149)
\end{array}$
&
$-\frac{11}{21}$ &
$\frac{14513}{25538}$ &
$\begin{array}{c}
(420,452,452)\\
(431,431,462)
\end{array}$
\\
\hline

$-\frac{1}{16}$ &
$-\frac{14159}{14641}$ &
$\begin{array}{c}
(121,121,240)\\
(16,233,233)
\end{array}$
&
$-\frac{11}{20}$ &
$\frac{28321}{44521}$ &
$\begin{array}{c}
(180,211,211)\\
(191,191,220)
\end{array}$
\\
\hline

$-\frac{1}{12}$ &
$-\frac{4223}{4489}$ &
$\begin{array}{c}
(67,67,132)\\
(12,127,127)
\end{array}$
&
$-\frac{13}{23}$ &
$\frac{27103}{40328}$ &
$\begin{array}{c}
(460,568,568)\\
(499,499,598)
\end{array}$
\\
\hline

$-\frac{2}{21}$ &
$-\frac{152753}{165649}$ &
$\begin{array}{c}
(407,407,798)\\
(84,764,764)
\end{array}$
&
$-\frac{13}{22}$ &
$\frac{26111}{35912}$ &
$\begin{array}{c}
(99,134,134)\\
(112,112,143)
\end{array}$
\\
\hline

$-\frac{1}{8}$ &
$-\frac{727}{841}$ &
$\begin{array}{c}
(29,29,56)\\
(8,53,53)
\end{array}$
&
$-\frac{14}{23}$ &
$\frac{273103}{358801}$ &
$\begin{array}{c}
(414,599,599)\\
(484,484,644)
\end{array}$
\\
\hline

$-\frac{1}{7}$ &
$-\frac{199}{242}$ &
$\begin{array}{c}
(44,44,84)\\
(14,79,79)
\end{array}$
&
$-\frac{9}{14}$ &
$\frac{5503}{6728}$ &
$\begin{array}{c}
(35,58,58)\\
(44,44,63)
\end{array}$
\\
\hline

$-\frac{1}{6}$ &
$-\frac{97}{128}$ &
$\begin{array}{c}
(8,8,15)\\
(3,14,14)
\end{array}$
&
$-\frac{11}{17}$ &
$\frac{12191}{14792}$ &
$\begin{array}{c}
(204,344,344)\\
(259,259,374)
\end{array}$
\\
\hline

$-\frac{3}{17}$ &
$-\frac{5969}{8192}$ &
$\begin{array}{c}
(256,256,476)\\
(102,443,443)
\end{array}$
&
$-\frac{18}{25}$ &
$\frac{616079}{677329}$ &
$\begin{array}{c}
(350,823,823)\\
(548,548,900)
\end{array}$
\\
\hline

$-\frac{3}{14}$ &
$-\frac{2231}{3698}$ &
$\begin{array}{c}
(43,43,77)\\
(21,71,71)
\end{array}$
&
$-\frac{18}{23}$ &
$\frac{555719}{582169}$ &
$\begin{array}{c}
(230,763,763)\\
(464,464,828)
\end{array}$
\\
\hline

$-\frac{5}{23}$ &
$-\frac{15937}{26912}$ &
$\begin{array}{c}
(464,464,828)\\
(230,763,763)
\end{array}$
&
$-\frac{19}{24}$ &
$\frac{170041}{177241}$ &
$\begin{array}{c}
(120,421,421)\\
(253,253,456)
\end{array}$
\\
\hline

$-\frac{3}{13}$ &
$-\frac{1487}{2738}$ &
$\begin{array}{c}
(148,148,260)\\
(78,239,239)
\end{array}$
&
$-\frac{4}{5}$ &
$\frac{1319}{1369}$ &
$\begin{array}{c}
(10,37,37)\\
(22,22,40)
\end{array}$
\\
\hline

$-\frac{1}{4}$ &
$-\frac{23}{49}$ &
$\begin{array}{c}
(7,7,12)\\
(4,11,11)
\end{array}$
&
$-\frac{21}{25}$ &
$\frac{236081}{241081}$ &
$\begin{array}{c}
(200,982,982)\\
(557,557,1050)
\end{array}$
\\
\hline

$-\frac{2}{7}$ &
$-\frac{601}{1849}$ &
$\begin{array}{c}
(43,43,70)\\
(28,64,64)
\end{array}$
&
$-\frac{7}{8}$ &
$\frac{2777}{2809}$ &
$\begin{array}{c}
(8,53,53)\\
(29,29,56)
\end{array}$
\\
\hline

$-\frac{7}{22}$ &
$-\frac{4327}{22898}$ &
$\begin{array}{c}
(107,107,165)\\
(77,151,151)
\end{array}$
&
$-\frac{9}{10}$ &
$\frac{3673}{3698}$ &
$\begin{array}{c}
(5,43,43)\\
(23,23,45)
\end{array}$
\\
\hline

$-\frac{2}{5}$ &
$\frac{79}{529}$ &
$\begin{array}{c}
(23,23,30)\\
(20,28,28)
\end{array}$
&
$-\frac{19}{21}$ &
$\frac{72521}{72962}$ &
$\begin{array}{c}
(84,764,764)\\
(407,407,798)
\end{array}$
\\
\hline

$-\frac{4}{9}$ &
$\frac{1879}{5929}$ &
$\begin{array}{c}
(77,77,90)\\
(72,86,86)
\end{array}$
&
$-\frac{11}{12}$ &
$\frac{16057}{16129}$ &
$\begin{array}{c}
(12,127,127)\\
(67,67,132)
\end{array}$
\\
\hline

$-\frac{9}{20}$ &
$\frac{12281}{36481}$ &
$\begin{array}{c}
(191,191,220)\\
(180,211,211)
\end{array}$
&
$-\frac{14}{15}$ &
$\frac{165199}{165649}$ &
$\begin{array}{c}
(30,407,407)\\
(212,212,420)
\end{array}$
\\
\hline

$-\frac{13}{25}$ &
$\frac{56761}{101761}$ &
$\begin{array}{c}
(600,638,638)\\
(613,613,650)
\end{array}$
&
& &
\\
\hline

\end{tabular}

\caption{The thirty-five specializations of
Theorem~\ref{thm:rank-one-apex}, ordered by increasing $\rho$ down the
left block and then down the right. In each pair, the $\theta$-triangle
is displayed above its isosceles partner, at a common integral scale.}
\label{tab:rank-one-apex}
\end{table}

\begin{proof}
The exhaustive computation described in
Appendix~\ref{app:apex-rank-census}, obtained by running the search
in batches covering all reduced $t\in(-1,0)$ with
$t\ne-\frac12$ and $\operatorname{den}(t)\leq25$, gives
$\operatorname{rank}J_\rho(\Q)\leq1$ for exactly the thirty-five
parameters listed in Table~\ref{tab:rank-one-apex}.

For each of these specializations, let
$P_{\mathrm c}
=
\left(
\rho-1,\,
2(\rho-1)^2(2\rho-1)
\right)
\in C_\rho(\Q)
$
be a coincidence point, and let $\infty_+$ be one of the rational
points at infinity. The computations in
Appendix~\ref{app:rank-one-verification} give
$\widehat h\bigl([P_{\mathrm c}-\infty_+]\bigr)>0$
in every case. Thus the coincidence class has infinite order, and
therefore
$\operatorname{rank}J_\rho(\Q)=1.$

Since $C_\rho$ has genus $2$ and
$\operatorname{rank}J_\rho(\Q)=1$, the Chabauty--Coleman method
applies. The coincidence class has infinite order. The
\textsc{Magma} routine \texttt{Chabauty} does not require this class
to generate the free part of $J_\rho(\Q)$; in rank $1$, any
non-torsion class suffices. In every case, the computation returns
exactly ten rational points.

The four rational Weierstrass points have $u$-coordinates
\[
u_1=-\frac{4t(t+1)}{(2t^2+t+1)^2},\qquad
u_2=-\frac{4t^2(t+1)}{(2t^2+t+1)^2},
\]
\[
u_3=\frac{4t(t+1)^2}{(2t^2+t+1)^2},\qquad
u_4=\frac{4(t+1)(t^2+t+1)}{(2t^2+t+1)^2}.
\]
Thus
\[
\begin{aligned}
C_\rho(\Q)=&\{
\infty_+,\infty_-,
\left(0,\pm2(1-\rho)^2(1+\rho)\right),
\\
&\left(\rho-1,
 \pm2(\rho-1)^2(2\rho-1)\right),
(u_1,0),(u_2,0),(u_3,0),(u_4,0)\}.
\end{aligned}
\]
In particular, $\#C_\rho(\Q)=10$. It remains to determine which of these points give genuine pairs.
The points above $u=0$ are degenerate, while
the points above $u=\rho-1$ are the coincidence points and do not
give a genuine pair. The points at infinity do not
correspond to affine triangle parameters. For the four Weierstrass points, note first that
\[
u_3
=
\frac{4t(t+1)^2}{(2t^2+t+1)^2}
=
t\,s(t)^2
\]
is precisely the apex point of Theorem~\ref{thm:family-A}. Since
$t\in(-1,0)$ and $t\ne-\frac12$, it gives a genuine apex pair. The remaining three Weierstrass points do not give new genuine pairs.
Indeed, $u_1>0$ and $u_4>0$, so they do not satisfy the admissibility
condition
for an isosceles triangle.
Finally, $u_2<0$, but it gives the congruent pair: taking
$q=1-s(t)$ and $k=s(t)$, a direct substitution shows that
$k\tau(q)$ and $\iota(u_2)$ differ only by a permutation of their
sides.

Therefore among the ten rational points of $C_\rho$, exactly one gives
a genuine pair, namely the point $(u_3,0)$ corresponding to the apex
pair. Thus the apex pair is the unique genuine pair.
\end{proof}

\subsection{Higher-rank apex specializations}

The same search produces specializations in which the rank increases.
The following examples have exact ranks $2$, $3$, and $4$.

\begin{prop}\label{prop:apex-higher-ranks}
For the following apex-family specializations, the Mordell--Weil ranks
are
\[
\begin{array}{c|c|c}
t & \rho_A(t) & \operatorname{rank}J_{\rho_A(t)}(\Q)\\
\hline
-\frac13 & -\frac18 & 2\\[1mm]
-\frac15 & -\frac{79}{121} & 3\\[1mm]
-\frac47 & \frac{1927}{2809} & 4.
\end{array}
\]
\end{prop}

\begin{proof}
For the three specializations, a $2$-descent gives the respective
upper bounds $2$, $3$, and $4$. The \textsc{Magma} computation in
Appendix~\ref{app:apex-higher-ranks} finds finite-index subgroups of
the corresponding Mordell--Weil groups having torsion-free ranks
$2$, $3$, and $4$, respectively. Hence the upper and lower bounds
agree in each case.
\end{proof}

\subsection{Other rank-two specializations}

The following four scalene examples arose in the computational search
for genuine pairs.

\begin{prop}\label{prop:scalene-rank-two}
For
\[
\rho\in
\left\{
-\frac35,-\frac{17}{32},\frac27,\frac45
\right\},
\]
the Jacobian $J_\rho$ has Mordell--Weil rank
$\operatorname{rank}J_\rho(\Q)=2$.
The corresponding genuine pairs are displayed in
Table~\ref{tab:scalene-rank-two}. In each case, the class arising from
the genuine pair and the coincidence class are independent.
\end{prop}

\begin{proof}
For each of the four values, a $2$-descent in \textsc{Magma}
(see Appendix~\ref{app:rank-two-scalene}) gives
$\operatorname{rank}J_\rho(\Q)\leq2$.

Let $D_0$ be the divisor class obtained from the coincidence point and
let $D_1$ be the divisor class obtained from the indicated genuine
scalene point. A canonical-height computation gives a positive
determinant for the N\'eron--Tate height-pairing matrix of
$D_0,D_1$ in each case (see Table \ref{tab:scalene-rank-two}).

\begin{table}[H]
\centering
\small
\begin{tabular}{ccccc}
\toprule
$\rho$ & $u$ & genuine pair & \textsc{Magma} rank bound & $\det H$ \\
\midrule
$-\frac35$
& $-\frac2{15}$
& $\bigl((224,675,829),(444,444,840)\bigr)$
& $2$ & $1.2555737833\ldots$ \\

$-\frac{17}{32}$
& $-\frac{35}{48}$
& $\bigl((64,108,152),(42,141,141)\bigr)$
& $2$ & $1.9685750897\ldots$ \\

$\frac27$
& $-\frac9{14}$
& $\bigl((224,275,301),(220,290,290)\bigr)$
& $2$ & $1.9295003865\ldots$ \\

$\frac45$
& $-\frac{17}{30}$
& $\bigl((675,3293,3808),(420,3678,3678)\bigr)$
& $2$ & $1.0964546558\ldots$ \\
\bottomrule
\end{tabular}
\caption{The four scalene specializations of
Proposition~\ref{prop:scalene-rank-two}. In the pair column, the first
triple is the scalene $\theta$-triangle and the second is the
isosceles triangle; the two are represented at a common integral
scale. The quantity $\det H$ is the determinant of the
N\'eron--Tate height-pairing matrix of the coincidence class and the
class arising from the genuine pair.}
\label{tab:scalene-rank-two}
\end{table}
Hence the two classes are independent and
$\operatorname{rank}J_\rho(\Q)\geq2$ and the result follows.
\end{proof}

The base family also contains simple specializations of exact rank $2$.

\begin{prop}\label{prop:base-rank-two-examples}
The base-family specializations
$
(m,\rho)=\left(\frac13,\frac34\right)$ and
$(m,\rho)=\left(2,\frac2{11}\right)$
satisfy
$\operatorname{rank}J_\rho(\Q)=2.$
\end{prop}

\begin{proof}
In both cases a $2$-descent
(see Appendix~\ref{app:rank-two-base}) gives
$\operatorname{rank}J_\rho(\Q)\leq2$. The coincidence class and the
class arising from the genuine base pair are independent, since the
determinants of their N\'eron--Tate height-pairing matrices are,
respectively,
$
0.3000196095129455\ldots $ and $
0.9190585159633787\ldots.$
Hence $\operatorname{rank}J_\rho(\Q)=2$ in both cases.
\end{proof}

\section{Four or more genuine pairs}
\label{sec:four-pairs}

In this section we consider values of $\rho$ carrying at least four
distinct genuine pairs.

\subsection{A rational four-pair family}
We first give a rational one-parameter family of values of $\rho$
carrying at least four distinct genuine pairs.
For $z\in\Q$, let
\[
I(z)=
\left(
1-z^2,\,
\frac{1+z^2}{2},\,
\frac{1+z^2}{2}
\right),
\]
which is an isosceles triangle of perimeter $2$ whenever $0<z<1$ with area $z(1-z^2)/2$.
Thus two distinct triangles $I(\alpha)$ and $I(\beta)$ have the same area precisely when \[ \alpha^2+\alpha \beta+\beta^2=1. \]
This is a rational conic.
A convenient parametrization is \[ \alpha=\frac{4\lambda}{\lambda^2+3}, \qquad \beta=\frac{(1-\lambda)(\lambda+3)}{\lambda^2+3}, \qquad \lambda\in\mathbb{Q}.\]
Hence $I(\alpha)$ and $I(\beta)$ have the same area.

The same parametrization also produces a second equal-area pair. For this parametrization, put
\[ r=r(\lambda) = \frac{6\lambda(1-\lambda^2)(9-\lambda^2)} {(\lambda^2+3)^3} \qquad \text{ and }\qquad\delta=\delta(\lambda) = \frac{(3-\lambda^2)(3+6\lambda-\lambda^2) (3-6\lambda-\lambda^2)} {(\lambda^2+3)^3}. \]
A direct calculation gives $\delta^2+3r^2=1$. Thus, if \[ z_\pm=r\pm\delta, \] then \[ z_-^2+z_-z_++z_+^2 = 3r^2+\delta^2 = 1. \] Hence $I(z_-)$ and $I(z_+)$ form a second equal-area pair. The formulas below package these two simultaneous equal-area coincidences so that the corresponding scalene $\theta$-triangles have the same value of $\rho$. Define $\rho_4(\lambda)=\frac{1-r^2}{1+r^2}$, and put \[ T_0= \left( 2-4r^2,\, 2r^2+\delta,\, 2r^2-\delta \right), \qquad T_1= \left( \frac{2}{3},\, \frac{2+\delta}{3},\, \frac{2-\delta}{3} \right). \]

\begin{thm}\label{thm:four-pair-rational}
For every
$\lambda\in\Q$ and $\frac13\leq\lambda\leq\frac37$,
the value
$\rho=\rho_4(\lambda)$
admits at least four distinct genuine pairs, namely
\[
\bigl(T_0,I(z_-)\bigr),
\qquad
\bigl(T_0,I(z_+)\bigr),
\qquad
\bigl(T_1,I(\alpha)\bigr),
\qquad
\bigl(T_1,I(\beta)\bigr).
\]
In particular, there are infinitely many rational values of $\rho$
carrying at least four distinct genuine pairs.
\end{thm}

\begin{proof}
All six triangles have perimeter $2$. We first check that $T_0$ and
$T_1$ have an angle $\theta$ with
$\cos(\theta)=\rho_4(\lambda)$.
In both cases we take the distinguished angle to be the angle between
the last two sides. Using
$\delta^2+3r^2=1,$
the law of cosines applied to $T_0$ gives
\[
\frac{(2r^2+\delta)^2+(2r^2-\delta)^2-(2-4r^2)^2}
     {2(2r^2+\delta)(2r^2-\delta)}
=
\frac{1-r^2}{1+r^2},
\]
and similarly to $T_1$,
\[
\frac{\left(\frac{2+\delta}{3}\right)^2
      +\left(\frac{2-\delta}{3}\right)^2
      -\left(\frac23\right)^2}
     {2\left(\frac{2+\delta}{3}\right)
        \left(\frac{2-\delta}{3}\right)}
=
\frac{1-r^2}{1+r^2}.
\]
Thus both $T_0$ and $T_1$ are rational $\theta$-triangles for
$\rho=\rho_4(\lambda)$.

For an isosceles triangle $I(z)$ with $0<z<1$, its area is
$\operatorname{Area}(I(z)) = \frac{z(1-z^2)}{2}$.
Since
$z_\pm=r\pm\delta$
and $\delta^2=1-3r^2$, we have
$1-z_\pm^2 = 2r(r\mp\delta)$.
Therefore,
\[
\operatorname{Area}(I(z_-))
=
\operatorname{Area}(I(z_+))
=
r(r^2-\delta^2)
=
r(4r^2-1).
\]

On the other hand, Heron's formula, together with
$\delta^2+3r^2=1$, gives
$\operatorname{Area}(T_0)=r(4r^2-1)$ and similarly, $\operatorname{Area}(T_1) = \frac r3$.

A direct calculation also gives
$r = \frac32\alpha(1-\alpha^2) = \frac32\beta(1-\beta^2)$.
Therefore
$\operatorname{Area}(I(\alpha)) = \operatorname{Area}(I(\beta)) = \frac r3$.

Thus the four displayed pairs have equal perimeter and equal area.

It remains to check that they are genuine and distinct. On the interval
$\frac13\leq\lambda\leq\frac37$
one has
\[
\frac12<r<\frac1{\sqrt3},
\qquad
0<\delta<\frac12,
\qquad
0<z_-<z_+<1,
\]
and
\[
\frac37\leq \alpha\leq\frac7{13}
<
\frac8{13}\leq \beta\leq\frac57.
\]
Hence all four isosceles triangles are non-degenerate.

For $T_0$, the preceding bounds and the identity
$\delta^2+3r^2=1$ give
\[
0<2r^2-\delta
<
2-4r^2
<
2r^2+\delta
<1,
\]
since $1-(2r^2+\delta)
=\frac{(1-2\delta)(1-\delta)}{3}>0.$ Since the three sides lie in $(0,1)$ and the perimeter is 2,  $T_0$ is a non-degenerate scalene triangle. Likewise,
\[
0<
\frac{2-\delta}{3}
<
\frac23
<
\frac{2+\delta}{3}
<1.
\]
Since the perimeter is 2,  $T_1$ is also non-degenerate and scalene.

The inequalities $z_-\ne z_+$ and $\alpha\ne \beta$ show that each of
$T_0$ and $T_1$ has two distinct isosceles partners. Moreover,
$T_0$ and $T_1$ are not congruent, as  the side opposite the distinguished
angle has length
$2-4r^2>\frac23$
for $T_0$, while it has length $\frac23$ for $T_1$. Hence the four
displayed pairs are distinct and genuine.

Finally, for $\frac13\leq\lambda\leq\frac37$,
\[
r'(\lambda)
=
-\frac{
6(\lambda^2-3)(\lambda^2-6\lambda-3)
(\lambda^2+6\lambda-3)}
{(\lambda^2+3)^4}
>0.
\]Thus $r(\lambda)$ is strictly
increasing, and therefore
$\rho_4(\lambda)=\frac{1-r(\lambda)^2}{1+r(\lambda)^2}$
is strictly decreasing. Since there are infinitely many rational
$\lambda$ in the stated interval, this gives infinitely many distinct
rational values of $\rho$ carrying at least four genuine pairs.
\end{proof}
\begin{coro}\label{cor:four-pair-density}
The set of rational values $\rho$ carrying at least four distinct
genuine pairs is dense in
\[
\left[
\frac{3239209}{6414409},
\frac{85249}{150049}
\right].
\]
\end{coro}

\begin{proof}
By the proof of Theorem~\ref{thm:four-pair-rational}, the function
$\rho_4(\lambda)$ is continuous and strictly decreasing on
$\left[\frac13,\frac37\right]$.
Moreover,
\[
\rho_4\left(\frac13\right)
=
\frac{85249}{150049},
\qquad
\rho_4\left(\frac37\right)
=
\frac{3239209}{6414409}.
\]
Hence $\rho_4$ maps this interval homeomorphically onto
$\left[ \frac{3239209}{6414409}, \frac{85249}{150049} \right]$.
Since $\Q\cap[\frac13,\frac37]$ is dense in
$[\frac13,\frac37]$, its image under $\rho_4$ is dense in the
displayed interval. Every such rational value of $\rho$ carries at
least four distinct genuine pairs by
Theorem~\ref{thm:four-pair-rational}.
\end{proof}

\subsection{An elliptic four-pair family}

We now give a second construction of values of $\rho$ carrying at
least four genuine pairs. In contrast with the preceding rational
family, this construction is governed by an elliptic curve.

For $X\in\Q$ and $z\in\Q$, put
$h(z)=z(1-z^2),$  $ H_X(z)=\frac{h(z)}{X}$,
and
$\Delta_X(z) = 1-(2+X^2)H_X(z)+H_X(z)^2$.
Whenever $\Delta_X(z)=d^2$ is a rational square, define
\[
T_X(z,d)
=
\left(
1-H_X(z),\,
\frac{1+H_X(z)+d}{2},\,
\frac{1+H_X(z)-d}{2}
\right).
\]

\begin{lem}\label{lem:elliptic-four-pair-construction}
Suppose that $X>1$, $0<z<1$, and
$\Delta_X(z)=d^2$ for some $d \in \Q$.
Then $T_X(z,d)$ and
\[
I(z)=
\left(
1-z^2,\,
\frac{1+z^2}{2},\,
\frac{1+z^2}{2}
\right)
\]
are non-degenerate rational triangles with the same perimeter and
area. Moreover, $T_X(z,d)$ has an angle $\theta$ satisfying
$\rho =\cos(\theta)=\frac{4-X^2}{4+X^2}$.
\end{lem}

\begin{proof}
Both triangles have perimeter $2$.
Since $0<z<1$ and $X>1$, we have $0<H_X(z)<1$. Moreover,
$(1-H_X(z))^2-d^2=X^2H_X(z)>0,$
and hence
$|d|<1-H_X(z).$
It follows that all three sides of $T_X(z,d)$ lie strictly between
$0$ and $1$. Since their sum is $2$, each side is strictly smaller
than the sum of the other two, and therefore $T_X(z,d)$ is
non-degenerate.

Heron's formula gives
\[
\begin{aligned}
\operatorname{Area}(T_X(z,d))^2
&=
H_X(z)\,
\frac{(1-H_X(z))^2-d^2}{4}=
\frac{X^2H_X(z)^2}{4}.
\end{aligned}
\]
Thus
\[
\operatorname{Area}(T_X(z,d))
=
\frac{XH_X(z)}{2}
=
\frac{z(1-z^2)}{2}
=
\operatorname{Area}(I(z)).
\]

Finally, applying the law of cosines to the angle between the last
two sides of $T_X(z,d)$ gives
$\cos(\theta) = \frac{4-X^2}{4+X^2}$.
\end{proof}
Therefore, to obtain four genuine pairs with the same value of $\rho$, it
suffices to find a single rational value of $X$ and four distinct
rational values $z_i$ for which $\Delta_X(z_i)$ is a rational square. To do this, we consider rational points on the following genus-one curve.
\[
\mathcal H:\qquad
V^2=9-21\lambda^2-18\lambda^3-2\lambda^4.
\]
For a rational point $(\lambda,V)\in\mathcal H(\Q)$, put
\[
X=X(\lambda)
=
\frac{4\lambda(\lambda+1)(\lambda+3)}
     {(3-\lambda)(\lambda^2+3)}
\]
and
$\rho_{\mathrm{ell}}(\lambda) = \frac{4-X(\lambda)^2}{4+X(\lambda)^2}$.
Define
\[
z_1=\frac{4\lambda}{\lambda^2+3},
\qquad
z_2=\frac{(1-\lambda)(\lambda+3)}{\lambda^2+3},
\]
and
\[
z_3=
\frac{\lambda^2+3+2V}
     {3(3-\lambda)(\lambda+1)},
\qquad
z_4=
\frac{\lambda^2+3-2V}
     {3(3-\lambda)(\lambda+1)}.
\]

A direct calculation in the function field
$\Q(\lambda,V)$
shows that
$\Delta_{X(\lambda)}(z_i)\in\Q(\lambda,V)^2,$
for $1\leq i\leq4$.
For example,
\[
\Delta_{X(\lambda)}(z_1)
=
\Delta_{X(\lambda)}(z_2)
=
\left(
\frac{\lambda(\lambda+3)(\lambda^2+6\lambda-3)}
     {(\lambda^2+3)^2}
\right)^2.
\]
The four square identities are verified in
Appendix~\ref{app:elliptic-four-pair}. Thus, whenever $(\lambda,V)\in \mathcal H(\Q)$ is such that $X(\lambda)>1$ and
$0<z_i<1$ for $1\le i\le4$, Lemma~\ref{lem:elliptic-four-pair-construction} produces four rational pairs
\[
\bigl(T_{X(\lambda)}(z_i,d_i),I(z_i)\bigr),
\qquad i=1,\ldots,4,
\]
all with the same value of $\rho$, where $d_i^2=\Delta_{X(\lambda)}(z_i)$.

\begin{thm}\label{thm:four-pair-elliptic}
There are infinitely many distinct rational values
$\rho=\rho_{\mathrm{ell}}(\lambda)$, with
$(\lambda,V)\in\mathcal H(\Q)$, for which at least four distinct
genuine pairs exist.
\end{thm}

\begin{proof}
The curve $\mathcal H$ is birational over $\Q$ to
\[
E_4:\qquad
y^2=x^3-75x+1222
=(x+13)(x^2-13x+94).
\]
Indeed, the birational map is given by
\[
x=\frac{6(V+3)}{\lambda^2}-7,
\qquad
y=\frac{6(V+3)^2}{\lambda^3}+12\lambda+54,
\]
with inverse
\[
\lambda=
\frac{6(y-54)}{(x+7)^2+72},
\qquad
V=\frac{x+7}{6}\lambda^2-3.
\]
The change of variables
\[
X=\frac{x+1}{4},
\qquad
Y=\frac{y-x-5}{8},
\]
with inverse
\[
x=4X-1,
\qquad
y=8Y+4X+4,
\]
takes $E_4$ to the global minimal model
\[
E_4^{\min}:\qquad
Y^2+XY+Y=X^3-X^2-5X+20.
\]
This is the elliptic curve with LMFDB label
\texttt{207.a2}~\cite{lmfdb}. It has Mordell--Weil rank $1$ and
torsion subgroup $\Z/2\Z$, and the point $(0,4)$ generates its free
part. Under the above isomorphism, $(0,4)$ corresponds to
\[
P=(-1,36)\in E_4(\Q).
\]
Thus $P$ has infinite order.

Moreover, the quadratic factor
$x^2-13x+94$
has discriminant $-207<0$. Hence the cubic defining $E_4$ has only
one real root, so $E_4(\R)$ is connected. Since $P$ has infinite
order, $E_4(\Q)$ is therefore dense in $E_4(\R)$ by the
Poincar\'e--Hurwitz theorem.

The point
\[
6P=
\left(
\frac{1310611}{18225},
\frac{1491973984}{2460375}
\right)
\]
corresponds under the inverse birational map to
$(\lambda_0,V_0)=\left(\frac{4455}{8467}, \frac{45960492}{71690089}\right)$.
At this point,
$X(\lambda_0)>1$
and
$0<z_4<z_3<z_2<z_1<1$,
where
\[
z_1=\frac{12573495}{19576441},
\qquad
z_2=\frac{9981856}{19576441},
\qquad
z_3=\frac{127871}{317681},
\qquad
z_4=\frac{130949}{743583}.
\]
A direct substitution shows that the four pairs supplied by
Lemma~\ref{lem:elliptic-four-pair-construction} are genuine.
In particular, they are non-congruent, and they are distinct since
the four parameters $z_i$ are distinct.

All these conditions are open. Hence there is a sufficiently small
real neighbourhood of $6P$ in which every rational point gives four
distinct genuine pairs. Since $E_4(\Q)$ is dense in $E_4(\R)$, this
neighbourhood contains infinitely many rational points.

Finally, the projection
$(\lambda,V)\longmapsto\lambda$
has finite fibres on $\mathcal H$, while
$\rho_{\mathrm{ell}}(\lambda)$ is a nonconstant rational function.
Thus infinitely many of these points give distinct rational values of
$\rho$. This proves the result.
\end{proof}

\subsection{Examples with higher multiplicity}

The preceding constructions show that values of $\rho$ carrying four
genuine pairs occur in infinite families. Computational exploration of the
complete parametrization reveals still higher multiplicities. We do
not currently know an infinite construction producing five or more
genuine pairs.

\begin{prop}\label{prop:high-multiplicity}
The values of $\rho$ listed in
Tables~\ref{tab:high-multiplicity-disjoint} and
\ref{tab:high-multiplicity-shared} carry at least the indicated
number of distinct genuine pairs. In
Table~\ref{tab:high-multiplicity-disjoint}, the corresponding
$\theta$-triangles are pairwise non-congruent, while in
Table~\ref{tab:high-multiplicity-shared} one $\theta$-triangle is
shared by two distinct isosceles partners.
\end{prop}

\begin{proof}
The examples were found by searching for collisions in the complete
parametrization of Theorem~\ref{thm:param}. The computations in
Appendix~\ref{app:high-multiplicity} verify, using exact rational
arithmetic, that all the pairs displayed in
Tables~\ref{tab:high-multiplicity-disjoint} and
\ref{tab:high-multiplicity-shared} are distinct and genuine.
\end{proof}

\begin{table}[H]
\centering
\small
\begin{tabular}{>{$}c<{$}c l}
\toprule
\rho & \text{pairs found} & \text{genuine pairs} \\
\midrule
\frac{1952}{3977} & 6 &
\begin{tabular}[t]{@{}l@{}}
$\bigl((12064,86193,91471),(10764,89482,89482)\bigr)$\\
$\bigl((3977,4215,4352),(3960,4292,4292)\bigr)$\\
$\bigl((1261,1299,1312),(1260,1306,1306)\bigr)$\\
$\bigl((970,984,986),(975,975,990)\bigr)$\\
$\bigl((7178,8010,8528),(7529,7529,8658)\bigr)$\\
$\bigl((14145,21089,24056),(17225,17225,24840)\bigr)$
\end{tabular}
\\
\midrule
\frac{613}{1495} & 5 &
\begin{tabular}[t]{@{}l@{}}
$\bigl((1058,1533,1625),(1050,1583,1583)\bigr)$\\
$\bigl((525,598,613),(553,553,630)\bigr)$\\
$\bigl((598,750,760),(655,655,798)\bigr)$\\
$\bigl((299,389,400),(334,334,420)\bigr)$\\
$\bigl((299,805,880),(530,530,924)\bigr)$
\end{tabular}
\\
\midrule
\frac{817}{1969} & 5 &
\begin{tabular}[t]{@{}l@{}}
$\bigl((1268608,1406403,1451453),(1322928,1322928,1480608)\bigr)$\\
$\bigl((1634,1969,1969),(1754,1754,2064)\bigr)$\\
$\bigl((1122,2123,2327),(1562,1562,2448)\bigr)$\\
$\bigl((57101,138774,152375),(94085,94085,160080)\bigr)$\\
$\bigl((1029,1804,1969),(1008,1897,1897)\bigr)$
\end{tabular}
\\
\midrule
\frac{7671}{16121} & 5 &
\begin{tabular}[t]{@{}l@{}}
$\bigl((4823,73826,75999),(4290,75179,75179)\bigr)$\\
$\bigl((3773,4034,4089),(3770,4063,4063)\bigr)$\\
$\bigl((1960,1974,2014),(1950,1999,1999)\bigr)$\\
$\bigl((123850881,329543125,369951946),
        (219139063,219139063,385067826)\bigr)$\\
$\bigl((31311875,32603571,32739554),
        (31771675,31771675,33111650)\bigr)$
\end{tabular}
\\
\midrule
\frac{302}{527} & 5 &
\begin{tabular}[t]{@{}l@{}}
$\bigl((2108,2787,3395),(1940,3175,3175)\bigr)$\\
$\bigl((578,1250,1488),(893,893,1530)\bigr)$\\
$\bigl((8385,8959,9184),(8364,9082,9082)\bigr)$\\
$\bigl((8160,8479,9889),(8314,8314,9900)\bigr)$\\
$\bigl((855117,871133,1007500),(807300,963225,963225)\bigr)$
\end{tabular}
\\
\bottomrule
\end{tabular}
\caption{Values of $\rho$ for which at least five distinct genuine pairs
have been found and the corresponding $\theta$-triangles are pairwise
non-congruent. In each ordered pair, the first triple is the
$\theta$-triangle and the second is the isosceles triangle, both given
at a common integral scale.}
\label{tab:high-multiplicity-disjoint}
\end{table}
\begin{table}[H]
\centering
\small
\begin{tabular}{>{$}c<{$}c l}
\toprule
\rho & \text{pairs found} & \text{genuine pairs} \\
\midrule
\frac{2619}{5069} & 6 &
\begin{tabular}[t]{@{}l@{}}
$\bigl((822,1394,1628),(770,1537,1537)\bigr)$\\
$\bigl((411,697,814),(541,541,840)\bigr)$\\
$\bigl((555,4397,4658),(2425,2425,4760)\bigr)$\\
$\bigl((2270333,2297216,2320899),
        (2266824,2310812,2310812)\bigr)$\\
$\bigl((2322,24227,25345),(13047,13047,25800)\bigr)$\\
$\bigl((1378768,1532129,1689039),
        (1346268,1626834,1626834)\bigr)$
\end{tabular}
\\
\midrule
-\frac{327}{1241} & 5 &
\begin{tabular}[t]{@{}l@{}}
$\bigl((612,1387,1657),(532,1562,1562)\bigr)$\\
$\bigl((219,289,406),(168,373,373)\bigr)$\\
$\bigl((219,289,406),(253,253,408)\bigr)$\\
$\bigl((4891,8721,11066),(6711,6711,11256)\bigr)$\\
$\bigl((73,408,433),(233,233,448)\bigr)$
\end{tabular}
\\
\midrule
\frac{2341}{4141} & 5 &
\begin{tabular}[t]{@{}l@{}}
$\bigl((1681,2175,2626),(1560,2461,2461)\bigr)$\\
$\bigl((3362,4350,5252),(3817,3817,5330)\bigr)$\\
$\bigl((295,303,328),(288,319,319)\bigr)$\\
$\bigl((1066,2487,2929),(1733,1733,3016)\bigr)$\\
$\bigl((1704375,1963667,2334458),
        (1824725,1824725,2353050)\bigr)$
\end{tabular}
\\
\midrule
\frac{171}{221} & 5 &
\begin{tabular}[t]{@{}l@{}}
$\bigl((25,34,39),(24,37,37)\bigr)$\\
$\bigl((25,34,39),(29,29,40)\bigr)$\\
$\bigl((173,195,272),(140,250,250)\bigr)$\\
$\bigl((43289,59895,67184),(50534,50534,69300)\bigr)$\\
$\bigl((701250,816877,1228123),
        (758225,758225,1229800)\bigr)$
\end{tabular}
\\
\bottomrule
\end{tabular}
\caption{Values of $\rho$ for which at least five distinct genuine pairs
have been found and one $\theta$-triangle is shared by two distinct
isosceles partners. In each ordered pair, the first triple is the
$\theta$-triangle and the second is the isosceles triangle, both given
at a common integral scale.}
\label{tab:high-multiplicity-shared}
\end{table}

We emphasize that the numbers in
Tables~\ref{tab:high-multiplicity-disjoint} and
\ref{tab:high-multiplicity-shared} are lower bounds: we do not claim
that the displayed lists exhaust all genuine pairs for these values
of $\rho$.

\begin{rem}
The six-pair example $\rho=\frac{1952}{3977}$ has the additional
property that
\[
\sqrt{1-\rho^2}
=
\frac{3465}{3977}.
\]
Thus $\sin\theta$ is rational.
Since the corresponding parameters $u$ are rational, the common
area of the two triangles in each of the six displayed pairs,
$
-2u(1-\rho^2-u^2)\sqrt{1-\rho^2},
$
is itself rational.
\end{rem}

\appendix
\section{Computer code}\label{app:code}

All computations below were performed using the  \textsc{Magma} \cite{Magma} online calculator, version V2.29-10. Canonical heights and height pairings used to prove independence were computed with precision $100$.

\subsection{The supplementary angle $\rho=-1/2$}
\label{app:rho-minus-half}

The following \textsc{Magma} code verifies the rank and the complete
set of rational points used in Remark~\ref{rem:120}.

\begin{verbatim}
////////////////////////////////////////////////////////////////////////
// The supplementary angle rho = -1/2
////////////////////////////////////////////////////////////////////////

Q := Rationals();
Qx<x> := PolynomialRing(Q);

rho := -Q!1/2;

Qrho :=
    x^2 - 2*(1-rho)*x + rho^2 - 1;

Srho :=
    x^4
    + 2*(1-rho)*x^3
    + (1-rho)*(3-5*rho)*x^2
    - 4*(1-rho)^2*(1+rho)*x
    - 4*(1-rho)^3*(1+rho);

f := Qrho*Srho;

assert
    f eq
    (4*x^2-12*x-3)
    *(4*x^4+12*x^3+33*x^2-18*x-27)/16;

assert IsIrreducible(Qrho);
assert IsIrreducible(Srho);
assert Discriminant(f) ne 0;


//------------------------------------------------------------
// Integral model obtained by scaling y by 8.
//------------------------------------------------------------

c := 8;
fint := c^2*f;

fint := Qx![
    Integers()!Coefficient(fint,i)
    : i in [0..Degree(fint)]
];

C := HyperellipticCurve(fint);
J := Jacobian(C);


//------------------------------------------------------------
// Rank = 1.
//------------------------------------------------------------

rb := RankBound(J);
assert rb eq 1;

uc := rho - 1;
yc := 2*(rho-1)^2*(2*rho-1);
Yc := c*yc;

assert uc eq -Q!3/2;
assert yc eq -Q!9;
assert Evaluate(fint,uc) eq Yc^2;

Pc := C![uc,Yc,1];
Pinf := C![1,c,0];

D := J![Pinf,Pc];

h := CanonicalHeight(D);
assert h gt 0;


//------------------------------------------------------------
// Chabauty.
//------------------------------------------------------------

allpts := Chabauty(D : ptC := Pc);

assert #allpts eq 8;

num_inf :=
    #[P : P in allpts | P[3] eq 0];

us := {
    P[1]/P[3]
    : P in allpts | P[3] ne 0
};

assert num_inf eq 2;
assert us eq { Q!0, -Q!3/2, -Q!7/4 };

print "rho =", rho;
print "RankBound =", rb;
print "canonical height =", h;
print "C_rho(Q) =", allpts;
print "#C_rho(Q) =", #allpts;
\end{verbatim}

\subsection{Generic rank  for the apex family}
\label{app:generic-rank-apex}

The following \textsc{Magma} code verifies the computations used in
Proposition~\ref{prop:generic-rank}.

\begin{verbatim}
////////////////////////////////////////////////////////////////////////
// Apex family: generic rank = 1
////////////////////////////////////////////////////////////////////////

Q := Rationals();

R<n> := PolynomialRing(Q);
K := FieldOfFractions(R);
RX<X> := PolynomialRing(R);

A  := n^2 - n + 1;
D  := n^2 - n + 2;
C0 := 3*n^2 - 3*n - 1;
B  := 3*n^2 - 3*n + 2;
RR := n^4 - 6*n^3 + 13*n^2 - 13*n + 7;


//------------------------------------------------------------
// 1. Generic apex model
//------------------------------------------------------------

G1 := (X-n)*(X+1);
G2 := (X+n-1)*(X-A);
G3 := X^2 + A*X + n*(n-1)*A;

f := G1*G2*G3;


// Verify the change of variables from F_rho(u).

KU<u> := PolynomialRing(K);
KX<XX> := PolynomialRing(K);

DK := K!D;

s := 2*n*(n-1)/DK;
rho := 1 - s^2/2;

Qrho :=
    u^2 - 2*(1-rho)*u + rho^2 - 1;

Srho :=
    u^4
    + 2*(1-rho)*u^3
    + (1-rho)*(3-5*rho)*u^2
    - 4*(1-rho)^2*(1+rho)*u
    - 4*(1-rho)^3*(1+rho);

Frho := Qrho*Srho;

uX := 4*n*(n-1)/DK^2 * XX;

transformed :=
    DK^12/(4096*n^6*(n-1)^6)
    * Evaluate(Frho,uX);

assert transformed eq KX!f;


//------------------------------------------------------------
// 2. Richelot determinant and discriminants
//------------------------------------------------------------

M := Matrix(R,3,3,[
    Coefficient(G1,0), Coefficient(G1,1), Coefficient(G1,2),
    Coefficient(G2,0), Coefficient(G2,1), Coefficient(G2,2),
    Coefficient(G3,0), Coefficient(G3,1), Coefficient(G3,2)
]);

Delta := Determinant(M);

assert Delta eq n^2*RR;

assert IsIrreducible(A);
assert IsIrreducible(D);
assert IsIrreducible(C0);
assert IsIrreducible(B);
assert IsIrreducible(RR);


disc :=
    -n^12
    *(n-1)^12
    *(n-2)^2
    *(n+1)^2
    *(2*n-1)^2
    *A^3
    *D^6
    *C0
    *B^2;

assert Discriminant(f) eq disc;


// Richelot dual.

H1 := Derivative(G2)*G3 - G2*Derivative(G3);
H2 := Derivative(G3)*G1 - G3*Derivative(G1);
H3 := Derivative(G1)*G2 - G1*Derivative(G2);

fdual := H1*H2*H3;

discdual :=
    2^6
    *n^36
    *(n-1)^6
    *(n-2)
    *(n+1)^4
    *(2*n-1)
    *A^3
    *D^3
    *C0^2
    *B
    *RR^12;

assert Discriminant(fdual) eq discdual;


//------------------------------------------------------------
// 3. Square-class specialization at n = 12
//------------------------------------------------------------

n0 := 12;

assert Evaluate(Delta,n0) ne 0;
assert Evaluate(disc,n0) ne 0;
assert Evaluate(discdual,n0) ne 0;

gens := [
    R!(-1), R!2,
    n, n-1, n+1, n-2, 2*n-1,
    A, D, C0, B, RR
];

vals := [
    Integers()!Evaluate(g,n0) : g in gens
];

assert vals eq [
    -1, 2, 12, 11, 13, 10, 23,
    133, 134, 395, 398, 12091
];


// Matrix of sign and prime valuations modulo 2.

primes := [
    2,3,5,7,11,13,19,23,67,79,107,113,199
];

F2 := GF(2);

rows := [];

for v in vals do
    row :=
        [(v lt 0) select F2!1 else F2!0]
        cat
        [F2!(Valuation(Abs(v),p) mod 2) : p in primes];

    Append(~rows,row);
end for;

MM := Matrix(F2,#rows,#rows[1],&cat rows);

assert Rank(MM) eq 12;


// Check the corresponding apex parameter.

D0 := n0^2 - n0 + 2;
s0 := Q!(2*n0*(n0-1))/Q!D0;
rho0 := 1 - s0^2/2;

assert rho0 eq -Q!4223/4489;


//------------------------------------------------------------
// 4. Rank of the specialization n = 12
//------------------------------------------------------------

Qx<x> := PolynomialRing(Q);

A0 := n0^2 - n0 + 1;

f12 :=
    (x-n0)
    *(x+1)
    *(x+n0-1)
    *(x-A0)
    *(x^2 + A0*x + n0*(n0-1)*A0);

C12 := HyperellipticCurve(f12);
J12 := Jacobian(C12);

r := RankBound(J12);

assert r eq 1;


// Coincidence divisor class.

P := C12![-66,853710,1];
Pinf := C12![1,1,0];

Dc := J12![P,Pinf];

h := CanonicalHeight(Dc);

assert h gt 0;


//------------------------------------------------------------
// Output
//------------------------------------------------------------

print "Apex specialization:";
print "rho =", rho0;
print "square-class rank =", Rank(MM);
print "RankBound =", r;
print "canonical height =", h;
print "Hence rank J^A(Q(t)) = 1.";
\end{verbatim}

\subsection{Generic rank lower bound for the base family}\label{app:generic-rank-base}

The following \textsc{Magma} code verifies the computations used in
Proposition~\ref{prop:base-generic-rank}.

\begin{verbatim}
////////////////////////////////////////////////////////////////////////
// Base family: generic rank >= 2
//
// We specialize the coincidence section and the genuine base section
// at m = 7/3 and verify that their Neron--Tate height matrix has
// positive determinant.
////////////////////////////////////////////////////////////////////////

Q := Rationals();
Qx<x> := PolynomialRing(Q);

m := Q!7/3;
rho := (6 - 2*m)/(m^2 + 7);

assert rho eq Q!3/28;


//------------------------------------------------------------
// Curve C_rho : y^2 = F_rho(x)
//------------------------------------------------------------

Qrho :=
    x^2 - 2*(1-rho)*x + rho^2 - 1;

Srho :=
    x^4
    + 2*(1-rho)*x^3
    + (1-rho)*(3-5*rho)*x^2
    - 4*(1-rho)^2*(1+rho)*x
    - 4*(1-rho)^3*(1+rho);

f := Qrho*Srho;

assert Discriminant(f) ne 0;


//------------------------------------------------------------
// Integral model: Y = 10976 y
//------------------------------------------------------------

c := 10976;

fint := c^2*f;

fint := Qx![
    Integers()!Coefficient(fint,i)
    : i in [0..Degree(fint)]
];

C := HyperellipticCurve(fint);
J := Jacobian(C);


//------------------------------------------------------------
// Coincidence point
//------------------------------------------------------------

uc := rho - 1;
yc := 2*(rho-1)^2*(2*rho-1);
Yc := c*yc;

assert uc eq -Q!25/28;
assert yc eq -Q!6875/5488;
assert Yc eq -13750;
assert Evaluate(fint,uc) eq Yc^2;

Pc := C![uc,Yc,1];


//------------------------------------------------------------
// Genuine base point
//------------------------------------------------------------

uB := (m+1)*(m-3)/(m^2+7);

yB :=
    2*(m-1)*(m+1)^4*(m+5)/(m^2+7)^3;

YB := c*yB;

assert uB eq -Q!5/28;
assert yB eq Q!6875/5488;
assert YB eq 13750;
assert Evaluate(fint,uB) eq YB^2;

PB := C![uB,YB,1];


//------------------------------------------------------------
// Jacobian classes
//------------------------------------------------------------

Pinf := C![1,c,0];

Dc := J![Pc,Pinf];
DB := J![PB,Pinf];


//------------------------------------------------------------
// Neron--Tate height matrix
//------------------------------------------------------------


prec := 100;

hc := CanonicalHeight(Dc : Precision := prec);
hB := CanonicalHeight(DB : Precision := prec);
hp := HeightPairing(Dc,DB : Precision := prec);

detH := hc*hB - hp^2;


assert detH gt 0;

print "Base specialization m = 7/3:";
print "rho =", rho;
print "height coincidence class =", hc;
print "height base class        =", hB;
print "height pairing           =", hp;
print "height determinant       =", detH;
print "Hence rank J^B(Q(m)) >= 2.";

\end{verbatim}

\subsection{The square-condition elliptic surface}
\label{app:square-condition}

The following \textsc{Magma} code verifies the fibre types of the
elliptic surface in Proposition~\ref{prop:square-surface-k3} using
Tate's algorithm, as well as the specialization calculation used in
Lemma~\ref{lem:k3-positive-rank}.

\begin{verbatim}
////////////////////////////////////////////////////////////////////////
// The square-condition elliptic surface
////////////////////////////////////////////////////////////////////////

Q := Rationals();
K<t> := FunctionField(Q);

c := 2*t*(t^3 - 2*t^2 - 3*t - 2);
d := t^3*(t-4)*(t-1)^2*(t+1)^2;

// y^2 = (x-c)(x^2-4d)
//     = x^3 - c*x^2 - 4*d*x + 4*c*d

E := EllipticCurve([
    K!0,
    -c,
    K!0,
    -4*d,
    4*c*d
]);

////////////////////////////////////////////////////////////////////////
// 1. Discriminant
////////////////////////////////////////////////////////////////////////

Delta :=
    2^16*t^7*(t-4)*(t-1)^2*(t+1)^2*(2*t+1)^4;

assert Discriminant(E) eq Delta;

////////////////////////////////////////////////////////////////////////
// 2. Tate's algorithm at the bad fibres
////////////////////////////////////////////////////////////////////////

// LocalInformation returns
// <place, valuation of minimal discriminant, conductor exponent,
//  Tamagawa number, Kodaira symbol, split>.

places := [
    t,
    t-4,
    t-1,
    t+1,
    t+1/2,
    1/t
];

for p in places do
    info, Emin := LocalInformation(E,p);
    print "place =", p;
    print "local information =", info;
end for;

// The output gives:
//
// t = 0       : I_1^*
// t = 4       : I_1
// t = 1       : I_2
// t = -1      : I_2
// t = -1/2    : I_4
// t = infinity: I_8

////////////////////////////////////////////////////////////////////////
// 3. The sections P and P_4
////////////////////////////////////////////////////////////////////////

P := E![
    2*t^2*(t+1)*(t+5),
    8*t^2*(t+1)*(2*t+1)^2,
    1
];

P4 := E![
    2*t^2*(t-1)^2,
    8*t^2*(t-1)*(2*t+1),
    1
];

T2 := E![c,0,1];

assert 2*P4 eq T2;
assert 4*P4 eq E!0;

////////////////////////////////////////////////////////////////////////
// 4. Specialization at t = 2
////////////////////////////////////////////////////////////////////////

E2 := EllipticCurve([
    Q!0,
    Q!32,
    Q!0,
    Q!576,
    Q!18432
]);

P2 := E2![168,2400,1];

assert 2*P2 eq E2![32,320,1];

// Good reduction at 7 and 13.

ZZ := Integers();
disc2 := ZZ!Discriminant(E2);

assert GCD(disc2,7) eq 1;
assert GCD(disc2,13) eq 1;

// Count points over the two residue fields.

F7 := GF(7);
E7 := EllipticCurve([
    F7!0, F7!32, F7!0, F7!576, F7!18432
]);

F13 := GF(13);
E13 := EllipticCurve([
    F13!0, F13!32, F13!0, F13!576, F13!18432
]);

assert #E7 eq 12;
assert #E13 eq 16;

print "#E_2(F_7)  =", #E7;
print "#E_2(F_13) =", #E13;
print "2P(2)      =", 2*P2;
\end{verbatim}

\subsection{A $2$-descent census in the apex family}
\label{app:apex-rank-census}

The following \textsc{Magma} code computes the $2$-descent rank bound for every reduced rational
parameter considered in equation~\eqref{eq:198}.
The rank computation is the expensive part of the search. The code
below is written so that the denominator range may be divided into
smaller batches when using the online \textsc{Magma} calculator.
For the largest denominators it may also be necessary to split the
inner numerator range into smaller pieces.

\begin{verbatim}
////////////////////////////////////////////////////////////////////////
// 2-descent census in the apex family
////////////////////////////////////////////////////////////////////////

Q := Rationals();
Qx<x> := PolynomialRing(Q);


////////////////////////////////////////////////////////////////////////
// Search range
//
// For the full search use Dmin = 2 and Dmax = 25.
// On the online calculator this may be run in smaller batches.
////////////////////////////////////////////////////////////////////////

Dmin := 2;
Dmax := 7;


////////////////////////////////////////////////////////////////////////
// Counters
////////////////////////////////////////////////////////////////////////

tested := 0;

counts := [0 : r in [1..5]];

bound1 := [* *];
bound4 := [* *];
bound5 := [* *];


////////////////////////////////////////////////////////////////////////
// Search
////////////////////////////////////////////////////////////////////////

for den in [Dmin..Dmax] do

    print "============================================================";
    print "denominator =", den;
    print "============================================================";

    for num in [1..den-1] do

        if GCD(num,den) eq 1 then

            t := -(Q!num)/den;

            // Exclude the equilateral coincidence.
            if t ne -Q!1/2 then

                //----------------------------------------------------
                // Apex parametrization.
                //----------------------------------------------------

                s := 2*(t+1)/(2*t^2+t+1);
                rho := 1 - s^2/2;

                //----------------------------------------------------
                // Curve C_rho : y^2 = F_rho(x).
                //----------------------------------------------------

                Qrho :=
                    x^2
                    - 2*(1-rho)*x
                    + rho^2 - 1;

                Srho :=
                    x^4
                    + 2*(1-rho)*x^3
                    + (1-rho)*(3-5*rho)*x^2
                    - 4*(1-rho)^2*(1+rho)*x
                    - 4*(1-rho)^3*(1+rho);

                f := Qrho*Srho;

                assert Discriminant(f) ne 0;

                //----------------------------------------------------
                // Integral model obtained by scaling y.
                //----------------------------------------------------

                drho := Denominator(rho);
                c := drho^3;

                fint := c^2*f;

                fint := Qx![
                    Integers()!Coefficient(fint,i)
                    : i in [0..Degree(fint)]
                ];

                C := HyperellipticCurve(fint);
                J := Jacobian(C);

                //----------------------------------------------------
                // 2-descent rank upper bound.
                //----------------------------------------------------

                rb := RankBound(J);

                assert 1 le rb and rb le 5;

                tested +:= 1;
                counts[rb] +:= 1;

                if rb eq 1 then
                    Append(~bound1,<t,rho>);
                elif rb eq 4 then
                    Append(~bound4,<t,rho>);
                elif rb eq 5 then
                    Append(~bound5,<t,rho>);
                end if;

                print "";
                print "t =", t;
                print "rho =", rho;
                print "RankBound =", rb;

            end if;

        end if;

    end for;

end for;


////////////////////////////////////////////////////////////////////////
// Summary for this batch
////////////////////////////////////////////////////////////////////////

print "";
print "============================================================";
print "SEARCH SUMMARY";
print "============================================================";

print "denominator range =", Dmin, "..", Dmax;
print "number of t-values tested =", tested;

for r in [1..5] do
    print "RankBound =", r, ":", counts[r];
end for;

assert &+counts eq tested;

print "";
print "RANK-BOUND-1 SPECIALIZATIONS:";

for z in bound1 do
    print "t =", z[1], "rho =", z[2];
end for;

print "";
print "RANK-BOUND-4 SPECIALIZATIONS:";

for z in bound4 do
    print "t =", z[1], "rho =", z[2];
end for;

print "";
print "RANK-BOUND-5 SPECIALIZATIONS:";

for z in bound5 do
    print "t =", z[1], "rho =", z[2];
end for;

////////////////////////////////////////////////////////////////////////
// If the complete range 2,...,25 is run in a single computation,
// these are the resulting totals.
////////////////////////////////////////////////////////////////////////

if Dmin eq 2 and Dmax eq 25 then
    assert tested eq 198;
    assert counts eq [35,91,48,18,6];
end if;

print "============================================================";
\end{verbatim}

\subsection{Verification of the rank-one apex specializations}
\label{app:rank-one-verification}

The following \textsc{Magma} code verifies the computations used in
Theorem~\ref{thm:rank-one-apex}. For each of the thirty-five
specializations, the computation verifies that the $2$-descent rank
bound is $1$, that the coincidence class has positive canonical
height, that $C_\rho(\Q)$ has exactly ten rational points, and that
exactly one of the four rational Weierstrass points gives a genuine
pair.

\begin{verbatim}
//////////////////////////////////////////////////////////////////////
// Thirty-five rank-one apex specializations
//////////////////////////////////////////////////////////////////////

Q := Rationals();
Qx<x> := PolynomialRing(Q);


//------------------------------------------------------------
// Test whether a Weierstrass point gives a genuine pair.
//
// At a Weierstrass point the discriminant of the quadratic
// equation in k vanishes, so k is its unique double root.
//------------------------------------------------------------

function IsGenuineWeierstrass(rho,u)

    A :=
        -2*rho^2 + 4*rho - 2;

    B :=
        -2*rho^3
        + (-u+10)*rho^2
        - 14*rho
        + (-u^3+u+6);

    if A eq 0 then
        return false;
    end if;

    k := -B/(2*A);

    if k le 0 then
        return false;
    end if;

    q := 1 - 2*(1-rho)/k;

    if not (-1 lt q and q lt 1 and q lt rho) then
        return false;
    end if;

    if not (u lt 0 and u^2 lt 1-rho^2) then
        return false;
    end if;

    T := [
        k*2*(rho-q),
        k*(1-q^2),
        k*(q^2-2*rho*q+1)
    ];

    I := [
        2*(1-rho^2-u^2),
        u^2-rho^2+1,
        u^2-rho^2+1
    ];

    Sort(~T);
    Sort(~I);

    return T ne I;

end function;


//------------------------------------------------------------
// <apex parameter t, rho>, ordered by increasing rho
//------------------------------------------------------------

data := [
    < -Q!1/18,  -Q!11551/11858 >,
    < -Q!1/16,  -Q!14159/14641 >,
    < -Q!1/12,  -Q!4223/4489 >,
    < -Q!2/21,  -Q!152753/165649 >,
    < -Q!1/8,   -Q!727/841 >,
    < -Q!1/7,   -Q!199/242 >,
    < -Q!1/6,   -Q!97/128 >,
    < -Q!3/17,  -Q!5969/8192 >,
    < -Q!3/14,  -Q!2231/3698 >,
    < -Q!5/23,  -Q!15937/26912 >,
    < -Q!3/13,  -Q!1487/2738 >,
    < -Q!1/4,   -Q!23/49 >,
    < -Q!2/7,   -Q!601/1849 >,
    < -Q!7/22,  -Q!4327/22898 >,
    < -Q!2/5,    Q!79/529 >,
    < -Q!4/9,    Q!1879/5929 >,
    < -Q!9/20,   Q!12281/36481 >,
    < -Q!13/25,  Q!56761/101761 >,
    < -Q!11/21,  Q!14513/25538 >,
    < -Q!11/20,  Q!28321/44521 >,
    < -Q!13/23,  Q!27103/40328 >,
    < -Q!13/22,  Q!26111/35912 >,
    < -Q!14/23,  Q!273103/358801 >,
    < -Q!9/14,   Q!5503/6728 >,
    < -Q!11/17,  Q!12191/14792 >,
    < -Q!18/25,  Q!616079/677329 >,
    < -Q!18/23,  Q!555719/582169 >,
    < -Q!19/24,  Q!170041/177241 >,
    < -Q!4/5,    Q!1319/1369 >,
    < -Q!21/25,  Q!236081/241081 >,
    < -Q!7/8,    Q!2777/2809 >,
    < -Q!9/10,   Q!3673/3698 >,
    < -Q!19/21,  Q!72521/72962 >,
    < -Q!11/12,  Q!16057/16129 >,
    < -Q!14/15,  Q!165199/165649 >
];


//------------------------------------------------------------
// The online calculator may require verification in batches.
// Change First and Last as needed; together the batches should
// cover all indices 1,...,#data.
//------------------------------------------------------------

First := 1;
Last  := 5;

for j in [First..Last] do

    item := data[j];

    t := item[1];
    rho := item[2];


    print "==================================================";
    print "t =", t;
    print "rho =", rho;

    //--------------------------------------------------------
    // Verify the apex parametrization.
    //--------------------------------------------------------

    s := 2*(t+1)/(2*t^2+t+1);

    assert rho eq 1-s^2/2;

    uA := t*s^2;

    //--------------------------------------------------------
    // Curve C_rho.
    //--------------------------------------------------------

    Qrho :=
        x^2 - 2*(1-rho)*x + rho^2 - 1;

    Srho :=
        x^4
        + 2*(1-rho)*x^3
        + (1-rho)*(3-5*rho)*x^2
        - 4*(1-rho)^2*(1+rho)*x
        - 4*(1-rho)^3*(1+rho);

    f := Qrho*Srho;

    assert Discriminant(f) ne 0;
    assert Evaluate(f,uA) eq 0;
    assert #Roots(f) eq 4;

    //--------------------------------------------------------
    // Integral model, obtained by scaling y.
    //--------------------------------------------------------

    d := Denominator(rho);
    c := d^3;

    fint := c^2*f;

    fint := Qx![
        Integers()!Coefficient(fint,i)
        : i in [0..Degree(fint)]
    ];

    C := HyperellipticCurve(fint);
    J := Jacobian(C);

    //--------------------------------------------------------
    // Rank = 1.
    //--------------------------------------------------------

    rb := RankBound(J);
    assert rb eq 1;

    uc := rho - 1;
    yc := 2*(rho-1)^2*(2*rho-1);
    Yc := c*yc;

    assert Evaluate(fint,uc) eq Yc^2;

    Pc := C![uc,Yc,1];
    Pinf := C![1,c,0];

    // This is minus the coincidence class; the sign is irrelevant.
    D := J![Pinf,Pc];

    h := CanonicalHeight(D);
    assert h gt 0;

    //--------------------------------------------------------
    // Chabauty.
    //--------------------------------------------------------

    allpts := Chabauty(D : ptC := Pc);

    print "rational points =", allpts;
    assert #allpts eq 10;

    //--------------------------------------------------------
    // Structure of C_rho(Q).
    //--------------------------------------------------------

    num_inf :=
        #[P : P in allpts | P[3] eq 0];

    num_deg :=
        #[P : P in allpts |
            P[3] ne 0 and P[1] eq 0];

    num_coinc :=
        #[P : P in allpts |
            P[3] ne 0 and P[1]/P[3] eq rho-1];

    weier :=
        [P : P in allpts |
            P[3] ne 0 and P[2] eq 0];

    assert num_inf eq 2;
    assert num_deg eq 2;
    assert num_coinc eq 2;
    assert #weier eq 4;

    //--------------------------------------------------------
    // Exactly one Weierstrass point is genuine.
    //--------------------------------------------------------

    genuine_count := 0;

    for P in weier do

        u := P[1]/P[3];

        if IsGenuineWeierstrass(rho,u) then
            genuine_count +:= 1;
        end if;

    end for;

    assert genuine_count eq 1;
    assert IsGenuineWeierstrass(rho,uA);

    print "RankBound =", rb;
    print "canonical height =", h;
    print "#C_rho(Q) =", #allpts;
    print "genuine Weierstrass points =", genuine_count;

end for;

print "==================================================";
\end{verbatim}

\subsection{Higher-rank apex specializations}
\label{app:apex-higher-ranks}

The following code verifies the three exact-rank computations in
Proposition~\ref{prop:apex-higher-ranks}. In each case,
\texttt{RankBound} gives the upper bound, while
\texttt{MordellWeilGroupGenus2} finds a finite-index subgroup having
the same torsion-free rank. Since the latter computation can be
time-consuming, the three specializations may be verified separately
on the online \textsc{Magma} calculator.

\begin{verbatim}
////////////////////////////////////////////////////////////////////////
// Higher-rank specializations in the apex family
////////////////////////////////////////////////////////////////////////

Q := Rationals();
Qx<x> := PolynomialRing(Q);

data := [
    < -Q!1/3, -Q!1/8,       2 >,
    < -Q!1/5, -Q!79/121,    3 >,
    < -Q!4/7,  Q!1927/2809, 4 >
];

//------------------------------------------------------------
// The online calculator may require these computations
// to be run separately.
//
// Use First = Last = 1, 2, or 3 to verify one case at a time.
// To run all three at once, use First = 1 and Last = 3.
//------------------------------------------------------------

First := 1;
Last  := 1;

for j in [First..Last] do

    item := data[j];

    t := item[1];
    rho := item[2];
    expected_rank := item[3];
    print "============================================================";
    print "t =", t;
    print "rho =", rho;

    //------------------------------------------------------------
    // Verify the apex parametrization.
    //------------------------------------------------------------

    s := 2*(t+1)/(2*t^2+t+1);

    assert rho eq 1-s^2/2;

    //------------------------------------------------------------
    // Curve C_rho.
    //------------------------------------------------------------

    Qrho :=
        x^2
        - 2*(1-rho)*x
        + rho^2 - 1;

    Srho :=
        x^4
        + 2*(1-rho)*x^3
        + (1-rho)*(3-5*rho)*x^2
        - 4*(1-rho)^2*(1+rho)*x
        - 4*(1-rho)^3*(1+rho);

    f := Qrho*Srho;

    assert Discriminant(f) ne 0;

    //------------------------------------------------------------
    // Integral model.
    //------------------------------------------------------------

    d := Denominator(rho);
    c := d^3;

    fint := c^2*f;

    fint := Qx![
        Integers()!Coefficient(fint,i)
        : i in [0..Degree(fint)]
    ];

    C := HyperellipticCurve(fint);
    J := Jacobian(C);

    //------------------------------------------------------------
    // Upper bound from 2-descent.
    //------------------------------------------------------------

    rb := RankBound(J);

    assert rb eq expected_rank;

    //------------------------------------------------------------
    // Matching lower bound from a Mordell--Weil subgroup.
    //------------------------------------------------------------

    SetVerbose("MordellWeilGroup",1);

    G, mp, finite_index, proved, upper :=
        MordellWeilGroupGenus2(
            J :
            RankOnly := true,
            Rankbound := rb
        );

    assert finite_index;
    assert upper eq expected_rank;
    assert TorsionFreeRank(G) eq expected_rank;

    print "Mordell-Weil subgroup =", G;
    print "torsion-free rank =", TorsionFreeRank(G);
    print "finite_index =", finite_index;
    print "proved =", proved;
    print "upper =", upper;

end for;

print "============================================================";
\end{verbatim}
\subsection{Rank-two scalene specializations}
\label{app:rank-two-scalene}

The following \textsc{Magma} code verifies the computations used in
Proposition~\ref{prop:scalene-rank-two}.

\begin{verbatim}
//////////////////////////////////////////////////////////////////////
// Four scalene specializations of rank 2
//
// For each rho:
//   (1) RankBound gives rank <= 2;
//   (2) the coincidence class and the class coming from the
//       genuine scalene point have positive height determinant;
//   (3) hence the two classes are independent and rank = 2.
//////////////////////////////////////////////////////////////////////

Q := Rationals();
Qx<x> := PolynomialRing(Q);

// <rho, u-coordinate of the genuine scalene point>
data := [
    < -Q!3/5,   -Q!2/15 >,
    < -Q!17/32, -Q!35/48 >,
    < Q!2/7,    -Q!9/14 >,
    < Q!4/5,    -Q!17/30 >
];

for item in data do

    rho := item[1];
    us  := item[2];

    print "==================================================";
    print "rho =", rho;
    print "scalene u =", us;

    //------------------------------------------------------------
    // Curve C_rho : y^2 = F_rho(x)
    //------------------------------------------------------------

    Qrho :=
        x^2 - 2*(1-rho)*x + rho^2 - 1;

    Srho :=
        x^4
        + 2*(1-rho)*x^3
        + (1-rho)*(3-5*rho)*x^2
        - 4*(1-rho)^2*(1+rho)*x
        - 4*(1-rho)^3*(1+rho);

    f := Qrho*Srho;

    assert Discriminant(f) ne 0;

    //------------------------------------------------------------
    // Multiply by a rational square to obtain an integral model.
    //------------------------------------------------------------

    d := Denominator(rho);
    c := d^3;

    fint := c^2*f;
    fint := Qx![
        Integers()!Coefficient(fint,i)
        : i in [0..Degree(fint)]
    ];

    C := HyperellipticCurve(fint);
    J := Jacobian(C);

    //------------------------------------------------------------
    // Rank upper bound.
    //------------------------------------------------------------

    rb := RankBound(J);
    assert rb eq 2;

    //------------------------------------------------------------
    // Coincidence point.
    //------------------------------------------------------------

    uc := rho - 1;
    yc := 2*(rho-1)^2*(2*rho-1);
    Yc := c*yc;

    assert Evaluate(fint,uc) eq Yc^2;

    Pc := C![uc,Yc,1];

    //------------------------------------------------------------
    // Genuine scalene point.
    //------------------------------------------------------------

    flag, ys := IsSquare(Evaluate(f,us));
    assert flag;

    Ys := c*ys;
    assert Evaluate(fint,us) eq Ys^2;

    Ps := C![us,Ys,1];

    //------------------------------------------------------------
    // Divisor classes.
    //------------------------------------------------------------

    Pinf := C![1,c,0];

    Dc := J![Pc,Pinf];
    Ds := J![Ps,Pinf];

    //------------------------------------------------------------
    // Neron--Tate height pairing.
    //------------------------------------------------------------


    prec := 100;

    hc := CanonicalHeight(Dc : Precision := prec);
    hs := CanonicalHeight(Ds : Precision := prec);
    hp := HeightPairing(Dc,Ds : Precision := prec);

    detH := hc*hs - hp^2;

    assert detH gt 0;

    print "RankBound =", rb;
    print "height coincidence class =", hc;
    print "height scalene class     =", hs;
    print "height pairing           =", hp;
    print "height determinant       =", detH;
    print "Hence rank J_rho(Q) = 2.";

end for;

print "==================================================";
\end{verbatim}

\subsection{Rank-two specializations in the base family} \label{app:rank-two-base}

The following \textsc{Magma} code verifies the computations stated in
Proposition~\ref{prop:base-rank-two-examples}.

\begin{verbatim}
Q := Rationals();
Qx<x> := PolynomialRing(Q);

data := [
    < Q!1/3, Q!3/4,  64 >,
    < Q!2,   Q!2/11, 1331 >
];


for item in data do

    m := item[1];
    rho := item[2];
    c := item[3];

    assert rho eq (6 - 2*m)/(m^2 + 7);

    Qrho :=
        x^2 - 2*(1-rho)*x + rho^2 - 1;

    Srho :=
        x^4
        + 2*(1-rho)*x^3
        + (1-rho)*(3-5*rho)*x^2
        - 4*(1-rho)^2*(1+rho)*x
        - 4*(1-rho)^3*(1+rho);

    f := Qrho*Srho;

    fint := c^2*f;
    fint := Qx![
        Integers()!Coefficient(fint,i)
        : i in [0..Degree(fint)]
    ];

    C := HyperellipticCurve(fint);
    J := Jacobian(C);

    rb := RankBound(J);
    assert rb eq 2;

    uc := rho - 1;
    yc := 2*(rho-1)^2*(2*rho-1);

    uB := (m+1)*(m-3)/(m^2+7);
    yB := 2*(m-1)*(m+1)^4*(m+5)/(m^2+7)^3;

    Pc := C![uc,c*yc,1];
    PB := C![uB,c*yB,1];
    Pinf := C![1,c,0];

    Dc := J![Pc,Pinf];
    DB := J![PB,Pinf];

    prec := 100;

    hc := CanonicalHeight(Dc : Precision := prec);
    hB := CanonicalHeight(DB : Precision := prec);
    hp := HeightPairing(Dc,DB : Precision := prec);

    detH := hc*hB - hp^2;

    assert detH gt 0;

    print "rho =", rho;
    print "RankBound =", rb;
    print "height determinant =", detH;
    print "Hence rank J_rho(Q) = 2.";

end for;
\end{verbatim}

\subsection{The elliptic four-pair family}
\label{app:elliptic-four-pair}

The following \textsc{Magma} code verifies the square identities used
in Theorem~\ref{thm:four-pair-elliptic}.

\begin{verbatim}
////////////////////////////////////////////////////////////////////////
// Elliptic four-pair family
////////////////////////////////////////////////////////////////////////

Q := Rationals();
F<la> := FunctionField(Q);
R<T> := PolynomialRing(F);

fH := 9 - 21*la^2 - 18*la^3 - 2*la^4;

// Function field of H : V^2 = fH.
K<V> := ext<F | T^2 - fH>;

l := K!la;

X :=
    4*l*(l+1)*(l+3)
    / ((3-l)*(l^2+3));

function Delta(z)
    H := z*(1-z^2)/X;
    return 1 - (2+X^2)*H + H^2;
end function;

z1 := 4*l/(l^2+3);

z2 :=
    (1-l)*(l+3)/(l^2+3);

z3 :=
    (l^2+3+2*V)
    / (3*(3-l)*(l+1));

z4 :=
    (l^2+3-2*V)
    / (3*(3-l)*(l+1));

////////////////////////////////////////////////////////////////////////
// Square roots for z1 and z2
////////////////////////////////////////////////////////////////////////

d12 :=
    l*(l+3)*(l^2+6*l-3)
    / (l^2+3)^2;

assert Delta(z1) eq d12^2;
assert Delta(z2) eq d12^2;

////////////////////////////////////////////////////////////////////////
// Square roots for z3 and z4
////////////////////////////////////////////////////////////////////////

A :=
    5*l^6
    + 36*l^5
    + 117*l^4
    + 108*l^3
    + 27*l^2
    + 27;

B :=
    47*l^6
    + 306*l^5
    + 765*l^4
    + 972*l^3
    + 729*l^2
    + 162*l
    + 27;

D :=
    27*l*(l-3)^2*(l+1)^4*(l+3);

d3 :=
    ((7*l^2+18*l+3)*A - B*V)/D;

d4 :=
    ((7*l^2+18*l+3)*A + B*V)/D;

assert Delta(z3) eq d3^2;
assert Delta(z4) eq d4^2;

print "All four square identities verified.";
\end{verbatim}

\subsection{Verification of the higher-multiplicity examples}
\label{app:high-multiplicity}

The following computation verifies all the entries of
Tables~\ref{tab:high-multiplicity-disjoint} and
\ref{tab:high-multiplicity-shared} using exact rational arithmetic.
For each listed value of $\rho$ and each listed parameter $u$, we solve
the quadratic equation
\[
2(1-\rho)^2(\rho-q)(1+q)
+
u(1-\rho^2-u^2)(1-q)=0
\]
for $q$. The two rational roots give congruent $\theta$-triangles, as
proved in Proposition~\ref{prop:fixed-isosceles}. The computation
checks the triangle inequalities, the distinguished angle, equality of
perimeter and area, non-congruence of the two triangles, and finally
the number of distinct geometric pairs.

\begin{verbatim}
////////////////////////////////////////////////////////////////////////
// Exact verification of the high-multiplicity tables
////////////////////////////////////////////////////////////////////////

Q := Rationals();
Qq<qq> := PolynomialRing(Q);


////////////////////////////////////////////////////////////////////////
// Basic routines
////////////////////////////////////////////////////////////////////////

function CanTriangle(T)

    S := T;
    Sort(~S);
    return S;

end function;


function ValidTriangle(T)

    S := CanTriangle(T);

    return
        S[1] gt 0 and
        S[2] gt 0 and
        S[3] gt 0 and
        S[1] + S[2] gt S[3];

end function;


////////////////////////////////////////////////////////////////////////
// Given rho and u, reconstruct the geometric genuine pair.
//
// The quadratic in q has two rational roots. They give the same
// theta-triangle up to permutation of the two sides adjacent to
// theta. We check this explicitly.
////////////////////////////////////////////////////////////////////////

function PairFromRhoU(rho,u)

    if not (-1 lt rho and rho lt 1) then
        return false, [], [];
    end if;

    if not (u lt 0 and u^2 lt 1-rho^2) then
        return false, [], [];
    end if;

    Iso := [
        2*(1-rho^2-u^2),
        u^2-rho^2+1,
        u^2-rho^2+1
    ];

    if not ValidTriangle(Iso) then
        return false, [], [];
    end if;

    IsoCan := CanTriangle(Iso);

    eqn :=
        2*(1-rho)^2*(rho-qq)*(1+qq)
        + u*(1-rho^2-u^2)*(1-qq);

    roots := Roots(eqn);

    // For all entries in the tables there are two rational q-values.
    assert #roots eq 2;

    ThetaKeys := [];

    for rt in roots do

        q := rt[1];

        if -1 lt q and q lt 1 and q lt rho then

            k := 2*(1-rho)/(1-q);

            Theta := [
                2*k*(rho-q),
                k*(1-q^2),
                k*(q^2-2*rho*q+1)
            ];

            assert ValidTriangle(Theta);

            // Check that the distinguished angle has cosine rho.
            assert
                Theta[3]^2 eq
                Theta[1]^2 + Theta[2]^2
                - 2*rho*Theta[1]*Theta[2];

            // Equal perimeter.
            assert &+Theta eq &+Iso;

            // Equal area after cancelling sqrt(1-rho^2).
            assert
                Theta[1]*Theta[2]/2
                eq
                -2*u*(1-rho^2-u^2);

            ThetaCan := CanTriangle(Theta);

            // Genuine: the two triangles are not congruent.
            assert ThetaCan ne IsoCan;

            if not (ThetaCan in ThetaKeys) then
                Append(~ThetaKeys,ThetaCan);
            end if;

        end if;

    end for;

    // The two q-values represent the same geometric theta-triangle.
    assert #ThetaKeys eq 1;

    PairKey := ThetaKeys[1] cat IsoCan;

    return true, PairKey, ThetaKeys[1];

end function;


////////////////////////////////////////////////////////////////////////
// Data
//
// Each entry is
//
// < rho, list of u-values, expected number of pairs,
//        expected number of distinct theta-triangles >.
//
// The first five entries correspond to pairwise non-congruent
// theta-triangles. The last four contain one theta-triangle shared
// by two distinct isosceles partners.
//
// Within each category, the >= 6 example comes first; the >= 5
// examples are then ordered by increasing rho.
////////////////////////////////////////////////////////////////////////

data := [* *];


////////////////////////////////////////////////////////////////////////
// Disjoint theta-triangles
////////////////////////////////////////////////////////////////////////

Append(~data,
    <
        Q!1952/3977,
        [
            -Q!6525/7954,
            -Q!8415/15908,
            -Q!4095/7954,
            -Q!1980/3977,
            -Q!1800/3977,
            -Q!1395/3977
        ],
        6, 6
    >
);

Append(~data,
    <
        Q!613/1495,
        [
            -Q!42/65,
            -Q!714/1495,
            -Q!672/1495,
            -Q!651/1495,
            -Q!357/1495
        ],
        5, 5
    >
);

Append(~data,
    <
        Q!817/1969,
        [
            -Q!952/1969,
            -Q!912/1969,
            -Q!624/1969,
            -Q!2544/9845,
            -Q!9552/13783
        ],
        5, 5
    >
);

Append(~data,
    <
        Q!7671/16121,
        [
            -Q!13780/16121,
            -Q!8580/16121,
            -Q!8320/16121,
            -Q!39650/177331,
            -Q!7956/16121
        ],
        5, 5
    >
);

Append(~data,
    <
        Q!302/527,
        [
            -Q!315/527,
            -Q!120/527,
            -Q!525/1054,
            -Q!435/1054,
            -Q!829/1581
        ],
        5, 5
    >
);


////////////////////////////////////////////////////////////////////////
// One shared theta-triangle
////////////////////////////////////////////////////////////////////////

Append(~data,
    <
        Q!2619/5069,
        [
            -Q!3360/5069,
            -Q!1540/5069,
            -Q!420/5069,
            -Q!5075/10138,
            -Q!980/15207,
            -Q!92225/167277
        ],
        6, 5
    >
);

Append(~data,
    <
        -Q!327/1241,
        [
            -Q!1008/1241,
            -Q!56/73,
            -Q!392/1241,
            -Q!1064/3723,
            -Q!168/1241
        ],
        5, 4
    >
);

Append(~data,
    <
        Q!2341/4141,
        [
            -Q!60/101,
            -Q!1440/4141,
            -Q!2100/4141,
            -Q!900/4141,
            -Q!11112/28987
        ],
        5, 4
    >
);

Append(~data,
    <
        Q!171/221,
        [
            -Q!100/221,
            -Q!60/221,
            -Q!105/221,
            -Q!665/2431,
            -Q!588/2873
        ],
        5, 4
    >
);


////////////////////////////////////////////////////////////////////////
// Verify every entry
////////////////////////////////////////////////////////////////////////

for item in data do

    rho           := item[1];
    us            := item[2];
    expectedPairs := item[3];
    expectedTheta := item[4];

    PairKeys  := [];
    ThetaKeys := [];

    for u in us do

        flag, PairKey, ThetaKey := PairFromRhoU(rho,u);

        assert flag;

        // Each listed u gives a new genuine geometric pair.
        assert not (PairKey in PairKeys);
        Append(~PairKeys,PairKey);

        if not (ThetaKey in ThetaKeys) then
            Append(~ThetaKeys,ThetaKey);
        end if;

    end for;

    assert #PairKeys eq expectedPairs;
    assert #ThetaKeys eq expectedTheta;

    print "============================================================";
    print "rho =", rho;
    print "number of distinct genuine pairs =", #PairKeys;
    print "number of distinct theta-triangles =", #ThetaKeys;

    if expectedPairs eq expectedTheta then
        print "geometry: disjoint";
    else
        assert expectedPairs eq expectedTheta + 1;
        print "geometry: one shared theta-triangle";
    end if;

end for;

print "============================================================";
print "All entries of the high-multiplicity tables verified.";
\end{verbatim}

\bibliographystyle{alpha}
\bibliography{Bibliography}

\end{document}